\documentclass[10pt]{amsart}
\usepackage{amssymb,amsbsy,amsmath,amsfonts,times, amscd}
\usepackage{latexsym,euscript,exscale}
\usepackage{graphicx}
\usepackage{pinlabel}
\usepackage{enumerate}

\usepackage{helvet}

\title{Large scale geometry of homeomorphism groups}
\author {Kathryn Mann and Christian Rosendal}
\address{Department of Mathematics, Statistics, and Computer Science (M/C 249)\\
University of Illinois at Chicago\\
851 S. Morgan St.\\
Chicago, IL 60607-7045\\
USA}
\email{rosendal.math@gmail.com}
\urladdr{http://homepages.math.uic.edu/$~$rosendal}

\address{Department of Mathematics, UC Berkeley\\
970 Evans hall \\
Berkeley, CA 94720 \\
USA}
\email{kpmann@math.berkeley.edu}
\urladdr{math.berkeley.edu/~mann}

\date {}
\newcommand{\norm}[1]{\lVert#1\rVert}

\newcommand {\B}{\mathbb B}

\newcommand {\A}{\mathbb A}

\newcommand {\N}{\mathbb N}

\newcommand {\R}{\mathbb R}
\newcommand {\Z}{\mathbb Z}

\newcommand {\T}{\mathbb T}

\newcommand {\D}{\mathbb D}

\renewcommand{\S}{\mathbb S}

\DeclareMathOperator{\SL}{SL}

\newcommand{\eps}{\epsilon}

\newcommand{\iso}{\cong}

\newcommand{\inj}{\hookrightarrow}

\newcommand{\saa}{\Rightarrow}

\newcommand{\til}{\rightarrow}

\newcommand {\del}{ \; \big| \;}

\newcommand {\ku} {\mathcal}

\newcommand{\inv}{^{-1}}
\newcommand{\id}{{\rm id}}

\newcommand {\e} {\exists}

\DeclareMathOperator{\Diff}{Diff}
\DeclareMathOperator{\Homeo}{Homeo}

\DeclareMathOperator{\supp}{supp}

\newtheorem{thm}{Theorem}
\newtheorem{cor}[thm]{Corollary}
\newtheorem{lemma}[thm]{Lemma}
\newtheorem{prop} [thm] {Proposition}
\newtheorem{defi} [thm] {Definition}

\newtheorem{prob}[thm]{Problem}
\newtheorem{quest}[thm]{Question}

\theoremstyle{definition}

\newtheorem{rem}[thm]{Remark}
\newtheorem{exa}[thm]{Example}

\thanks{C. Rosendal was partially supported by a Simons Foundation Fellowship (Grant \#229959) and also recognizes support from the NSF (DMS 1201295).  K. Mann was supported by NSF grant \#0932078 while in residence at the MSRI during spring 2015.}

\begin{document}

\begin{abstract} 
Let $M$ be a compact manifold.  We show the identity component $\Homeo_0(M)$ of the group of self-homeomorphisms of $M$ has a well-defined quasi-isometry type, and study its large scale geometry.   Through examples, we relate this large scale geometry to both the topology of $M$ and the dynamics of group actions on $M$.  This gives a rich family of examples of non-locally compact groups to which one can apply the large-scale methods developed in previous work of the second author.  
\end{abstract}

\maketitle

\tableofcontents


\section{Introduction}
The aim of the present paper is to apply the large scale geometric methods developed in \cite{large scale} for general, i.e., not necessary locally compact, topological groups to the study of the homeomorphism groups of compact manifolds.    These homeomorphism groups are never locally compact, so the standard tools of geometric group theory (i.e. for finitely generated, or for locally compact, $\sigma$-compact groups) do not apply.  
A historical solution to this problem, though not our approach, is to consider instead the discrete \emph{mapping class} or \emph{homeotopy group}, the quotient of the group of orientation--preserving homeomorphisms of a manifold $M$ by the connected component of the identity.  In some cases, such as when $M$ is a compact surface, this mapping class group is known to be finitely generated and its large scale geometry understood.  In other cases, the mapping class group is not finitely generated, or no presentation is known.  In any case, a great deal of information on $\Homeo(M)$ is obviously lost by passing to this quotient.   Here we propose a study of the geometry of the identity component $\Homeo_0(M)$ itself.  

In \cite{large scale}, a framework is developed to generalizes geometric group theory in a broad setting.  This framework applies to any topological group $G$ admitting a \emph{maximal metric}, i.e. a compatible right-invariant metric $d$ such that, for any other compatible right-invariant metric $d'$, the identity map $(G, d) \to (G, d')$ is Lipschitz for large distances.   Up to quasi-isometry, such a metric may equivalently be given as the \emph{word metric} with respect to a suitable, namely \emph{relatively (OB)}, generating set for $G$.   Provided such a metric exists, it gives $G$ a canonical quasi-isometry type, and many of the standard tools of geometric group theory can be adapted to its study.   More details of this framework are recalled in Section \ref{prelim sec}.  
\bigskip

\noindent \textbf{General results. }
Our first result is a particularly nice description of a maximal metric for $\Homeo_0(M)$.  Let $\mathcal{B}$ be a finite open cover of $M$ by homeomorphic images of balls.  Define the \emph{fragmentation metric with respect to $\mathcal{B}$} on $\Homeo_0(M)$ by 
$$d_\mathcal{B}(f, g) : = \min \{k \mid fg^{-1} = f_1 f_2 \cdots f_k \}$$ 
where each $f_i$ is required to fix pointwise the complement of some ball $B \in \mathcal{B}$.   This is indeed a finite, right-invariant metric, moreover, we have the following. 

\begin{thm} \label{loc OB}
Let $M$ be a compact manifold. Then $\Homeo_0(M)$ admits a maximal metric  and hence  has a well defined quasi-isometry type. Moreover, this quasi-isometry type is simply that given by the  
\emph{fragmentation metric} with respect to any finite cover of $M$ by open balls. 
\end{thm}

We also show that the large-scale geometry of $\Homeo_0(M)$ reflects both the topology of $M$ and the dynamics of group actions on $M$.   For instance, on the topological side, we have the following theorem, strengthening earlier results of \cite{OB} and \cite{CF}.  

\begin{thm}[Compare with \cite{CF}]
For any $n$, $\Homeo_0(\S^n)$ is bounded, i.e. quasi-isometric to a point.   By contrast, as soon as $\dim(M) > 1$ and $M$ has infinite fundamental group, $\Homeo_0(M)$ contains a coarsely embedded copy of the Banach space $C([0,1])$; in particular, $\Homeo_0(M)$ is coarsely universal among all separable metric spaces
\end{thm}

On the dynamical side, ``distorted" subgroups of homeomorphism groups are known to have interesting dynamical properties (see e.g. \cite{CF}, \cite{FH}, \cite{Militon2}, \cite{Polterovich}), and our large-scale framework gives a much more natural definition of \emph{distortion}. Namely, a subgroup $G \leqslant \Homeo_0(M)$ is distorted if the inclusion $G \hookrightarrow \Homeo_0(M)$ is not a quasi-isometric embedding. 
\bigskip

\noindent \textbf{Extensions and covers.}
In Section \ref{covers sec}, we discuss the relationship between the large scale geometry of $\Homeo_0(M)$ and the group of \emph{lifts} of such homeomorphisms to a cover of $M$.   The question of lifts to the universal cover of $M$ arises naturally in computations involving the fragmentation metric.  We show that the group of all lifts to any cover also admits a maximal metric, hence well defined quasi-isometry type, and give conditions on when the group of lifts to the universal cover is quasi-isometric to the product $\pi_1(M) \times \Homeo_0(M)$.

\begin{thm} \label{section thm}
Let $M$ be a manifold.  There is a natural, central subgroup $A \leqslant  \pi_1(M)$ such that, whenever the quotient map $\pi_1(M) \to \pi_1(M)/A$ admits a bornologous section, then the group of lifts of homeomorphisms to $\tilde{M}$ is quasi-isometric to the product $\Homeo_0(M) \times \pi_1(M)$.
\end{thm} 
The subgroup $A$ is described in Section \ref{covers sec}.
In essence, Theorem \ref{section thm} shows that the large-scale geometry of the fundamental group of a manifold is reflected in the large scale geometry of the group of $\pi_1$--equivariant homeomorphisms of its universal cover.  We illustrate this with several examples.

The main tool in the proof of Theorem \ref{section thm} is a very general result on existence of bornologous sections for quotients of topological groups.

\begin{prop}
Suppose $G$ is a topological group generated by two commuting subgroups $K$ and $H$, where $K$ is coarsely embedded, and the quotient map $\pi\colon G\til G/K$ admits a bornologous section  $\phi\colon G/K\til H$. Then $K\times G/K$ is coarsely equivalent to $G$ via the map
$$
(x,f)\mapsto x\phi(f).
$$
\end{prop}  

As a consequence, this gives a generalization of, and in fact a new proof of, a theorem of Gersten \cite{Gersten}.  We also discuss related notions of \emph{bounded cohomology} and \emph{quasimorphisms} in this general context.  
\bigskip

\noindent \textbf{Groups acting on manifolds.}
Finally, Section \ref{group actions sec} discusses the dynamics of group actions from the perspective of large scale geometry.  While it remains an open problem to find any example of a finitely generated, torsion-free group that does not embed into $\Homeo_0(M)$ -- for any fixed manifold $M$  of dimension at least two -- restricting to \emph{quasi-isometric embeddings} makes the problem significantly more tractable.   Thus, we propose the study of quasi-isometric embeddings as new approach to the ``$C^0$ Zimmer program".   As evidence that this approach should be fruitful, we prove the following.  

\begin{thm} 
There exists a torsion-free, compactly generated group $G$ and manifold $M$, such that $G$ embeds continuously in $\Homeo(M)$, but $G$ does not admit any continuous, quasi-isometric isomorphic embedding in $\Homeo_0(M)$.  
\end{thm}

The proof of this theorem uses \emph{rotation sets} as a way to link the dynamics of group actions with the geometry of the group.   Although the group $G$ constructed in the proof is not very complicated, it does contain a copy of $\R$.  It would be very interesting to see a discrete example in its place.


\section{Preliminaries} \label{prelim sec}

\subsection{Coarse structures on topological groups} 

As is well-known, every topological group $G$ has a natural right-invariant uniform structure, which by a characterization of A. Weil is that induced by the family of all continuous right-invariant {\em \'ecarts} $d$ on $G$, i.e., metrics except that we may have $d(g,f)=0$ for $g\neq f$. In the same way, we may equip $G$ with a canonical right-invariant  coarse structure.  Here a {\em coarse space} (in the sense of \cite{roe}) 
is a set $X$ equipped with a family $\ku E$ of subsets $E\subseteq X\times X$ (called {\em coarse entourages}) that is closed under finite unions, taking subsets, inverses,  concatenation as given by $E\circ F=\{(x,y)\del \e z\; (x,z)\in E\;\&\; (z,y)\in F\}$, and so that the diagonal $\Delta$ belongs to $\ku E$. For example, if $d$ is an \'ecart on $X$, the associated coarse structure $\ku E_d$ is that generated by the basic entourages
$$
E_\alpha=\{(x,y)\del d(x,y)<\alpha\},
$$
for $\alpha<\infty$. Thus, if $G$ is any topological group, we may define its {\em right-invariant coarse structure} $\ku E_R$ by
$$
\ku E_R=\bigcap_d\ku E_d,
$$
where the intersection ranges over all continuous right-invariant \'ecarts $d$ on $G$.

Associated with the coarse structure $\ku E_R$ is the ideal of coarsely bounded subsets. Namely, a subset $A\subseteq G$ is said to be {\em relatively (OB) in $G$} if ${\rm diam}_d(A)<\infty$ for all continuous right-invariant \'ecarts $d$ on $G$. In the case that $G$ is {\em Polish}, i.e., separable and completely metrisable, this is equivalent to the following definition. 
\begin{defi} \label{OB def}
A subset $A\subseteq G$ is {\em relatively (OB) in $G$} if, for every identity neighbourhood $V\subseteq G$, there is a finite set $F \subseteq G$, and $n$ so that $A\subseteq (FV)^n$
\end{defi}
If $G$ is connected, the hypothesis on the finite set $F$ may be omitted as $F$ may be absorbed into a power of $V$. 
``OB" here stands for \emph{orbites born\'ees}, referring to the fact that for any continuous isometric action of $G$ on a metric space, the orbit of any point under a relatively (OB) subset of $G$ has finite diameter. 

For a Polish group $G$, having a relatively (OB) identity neighbourhood $V\subseteq G$ is equivalent to the condition that the coarse structure $\ku E_R$ is given by a single compatible right-invariant metric $d$, in the sense that $\ku E_R=\ku E_d$.  
In this case, we say that the metric $d$ is {\em coarsely proper}. Moreover, a compatible right-invariant metric $d$ is coarsely proper exactly when the  $d$-bounded subsets coincide with the relatively (OB) sets in $G$. If furthermore $G$ is {\em (OB) generated}, that is generated by a relatively (OB) subset $A\subseteq G$, we may define the corresponding right-invariant word metric $\rho_A$ on $G$ by
$$
\rho_A(g,f)=\min( k\del \e a_1,\ldots, a_k\in A\colon g=a_1\cdots a_kf).
$$
In this case, $G$ admits a compatible right-invariant metric $d$ that is {\em quasi-isometric} to the word metric $\rho_A$, i.e., so that
$$
\frac 1K\cdot \rho_A-C\leqslant d\leqslant K\cdot \rho_A+C
$$
for some constants $K \geq 1$ and $C \geq 0$. Such a metric $d$ is said to be {\em maximal} and is, up to quasi-isometry, independent of the choice of $A$. In fact, $A$ can always be taken to be an identity neighbourhood in $G$.

In contradistinction to the case of locally compact groups, many large Polish groups happen to have {\em property (OB)}, i.e., are relatively (OB) in themselves. This simply means that they have finite diameter with respect to every compatible right-invariant metric and therefore are quasi-isometric to a point. This can be seen as a metric generalization of compactness, as, within the class of locally compact second countable groups, this simply characterizes compactness.

Once this maximal metric, quasi-isometric to the word metrics $\rho_V$ of all relatively (OB)  identity neighbourhoods, is identified, the usual language and toolset of geometric group theory can be applied almost ad verbum, though occasionally additional care is required. For example, a map $\phi\colon G\til H$ between Polish groups is {\em bornologous} if, for every coarse entourage $E$ on $G$, the subsets $(\phi\times \phi)E$ give a coarse entourage for $H$. If $d_G$ and $d_H$ are coarsely proper metrics on $G$ and $H$, this simply means that, for every $\alpha<\infty$, there is $\beta<\infty$ so that 
$$
d_G(g,f)<\alpha\saa d_H(\phi(g),\phi(f))<\beta.
$$

We should note that, in the present paper, we have chosen to work with the right-invariant coarse structure $\mathcal E_R$, as opposed to the left-invariant coarse structure $\mathcal E_L$ defined analogously. This is due to the fact that the most natural metrics on homeomorphism groups are right-invariant. (These are the maximal displacement metrics, $d_\infty$, defined in the next section.) Moreover, this choice causes little conflict with the framework of \cite{large scale}, since the inversion map $g\mapsto g^{-1}$ is an isomorphism of the two coarse structures. 


\subsection{The local property (OB) for homeomorphism groups}

The goal of this section is to prove the following theorem, which is our starting point for the study of the large scale geometry of homeomorphism groups.  
\begin{thm} \label{Homeo OB}
Let $M$ be a compact manifold.  Then $\Homeo_0(M)$ has the local property (OB), hence a well defined quasi-isometry type.
\end{thm}  

Before we begin the proof, we recall some basic facts on the topology of $\Homeo(M)$.    If $M$ is any compact metrisable space, the 
{\em compact-open topology} on $\Homeo_0(M)$ is given by the subbasic open sets
$$
O_{K, U} = \{h\in \Homeo(M)\mid h(K) \subseteq U\},
$$
where $K\subseteq M$ is compact and $U\subseteq M$ open.
This topology is complete and separable, making $\Homeo(M)$ a Polish group. More generally, if $X$ is a locally compact metrizable space and $X^*$ its Alexandrov compactification, we can identify $\Homeo(X)$ with the pointwise stabilizer of the point at infinity $*$ inside of $\Homeo(X^*)$. Being a closed subgroup, $\Homeo(X)$ is itself a Polish group in the induced topology, which is that given by subbasic open sets
$$
O_{C,U}=\{h\in \Homeo(M)\mid h(C)\subseteq U\},
$$
where $C\subseteq X$ is closed and $U\subseteq X$ open and at least one of $C$ and $X\setminus U$ is compact. As shown by R. Arens \cite{arens}, when $X$ is locally connected, e.g., a manifold, this agrees with the compact-open topology on $\Homeo(X)$. All other transformation groups encountered here will be given the respective induced topologies.

If $d$ is a compatible metric on a compact metrizable space $M$, the compact-open topo\-logy on $\Homeo(M)$ is induced by the right-invariant metric
$$
d_\infty(g,f)=\sup_{x\in M}d(g(x),f(x)).
$$
Observe that, if instead $d$ is a compatible metric on a locally compact metrisable space $X$, the corresponding formula will not, in general, define a metric on $\Homeo(X)$, as distances could be infinite. 

It will often be useful to work with geodesic metrics and their associated lifts. Let us first recall that, by a result of R. H. Bing \cite{Bing},  every compact topological manifold $M$ admits a compatible geodesic metric, so we lose no generality working with these.  Moreover, for every such geodesic metric on $M$ and cover $\tilde{M} \overset{\pi}\to M$, there is a unique compatible geodesic metric $\tilde d$ on  $\tilde M$ so that the covering map $\pi\colon (\tilde{M}, \tilde{d}) \to (M, d)$ is a local isometry.  See e.g. \cite[sect. 3.4]{BII}.
This metric $\tilde d$ will be call the {\em geodesic lift} of $d$.

In particular, if $\tilde{M} \overset{\pi}\to M$ is the universal cover of a compact manifold $M$ and $d$ is a compatible geodesic metric on $M$ with geodesic lift $\tilde d$, then the action $\pi_1(M)\curvearrowright \tilde M$ by deck transformations preserves the metric $\tilde d$. As the action is also properly discontinuous and cocompact, the Milnor--Schwarz Theorem implies that, for every fixed $x\in \tilde M$, 
$$
\gamma\in \pi_1(M)\mapsto \gamma(x)\in \tilde M
$$
defines a quasi-isometry between the finitely generated discrete group $\pi_1(M)$ and the metric space $(\tilde M,\tilde d)$.

The proof of Theorem \ref{Homeo OB} relies on the following two lemmas.  Here, and in later sections, we use the standard notation $\supp(g)$ for the \emph{support} of an element $g \in \Homeo(M)$; this is the closure of the set $\{x \in M \mid g(x) \neq x\}$.

\begin{lemma}  \label{Bn lemma}
The group  $\Homeo_\partial(\B^n)$ of homeomorphisms of the $n$-ball that fix the boundary pointwise has property (OB).
\end{lemma}

\begin{proof}     
We equip $\B^n$ with the usual euclidean metric $d$.   Let $V$ be a neighborhood of the identity in $\Homeo_\partial(\B^n)$.  Then there exists $\eps$ such that  
$$
V_\eps := \{g\in \Homeo_\partial(\B^n)\del \sup_{x\in {\B^n}}d(gx,x)<\eps\} \subseteq V.
$$

Given $g\in \Homeo_\partial(\B^n)$,  the Alexander trick produces an isotopy $g_t$ in $\Homeo_\partial(\B^n)$ from $g_0 = \id$ to $g_1 = g$ via
$$
g_t(x)=\begin{cases}
x &\text{ if } t\leqslant \norm x\leqslant 1,\\
tg(\frac xt)&\text{ if }\norm x\leqslant t.
\end{cases}
$$
Since $\lim_{t\til 1}g_t=g$, there is some $t_0 < 1$ such that $g \in g_{t_0} V$.    Choose  $h \in V_\eps$ such that $h\left(t_0 \B^n \right)  \subseteq (1-\eps)\B^n$.  Since $\supp(g_{t_0}) \subseteq t_0 \B^n$, we have 
$$\supp \left( hg_{t_0}h\inv \right)= h \left( \supp(g_{t_0}) \right) \subseteq (1-\eps)\B^n.$$ 
Now choose $f\in \Homeo_\partial(\B^n)$ such that $f \left( (1-\eps)\B^n \right) \subseteq \frac{\eps}{2}\B^n$.  A similar argument shows that $\supp \left(f h g_{t_0} h\inv f\inv \right)\subseteq \frac \eps2\B^n$.   
As the latter set has diameter $\eps$, it follows that $fhg_{t_0}h\inv f\inv\in V$, hence $g_{t_0} \in V_\eps f\inv V f V_\eps$ and 
$$
g\in g_{t_0}V \subseteq  V  f\inv V fV^2 .
$$
Since $g$ was arbitrary, and $f$ independent of $g$, this shows that $\Homeo_\partial(\B^n)$ has property (OB), taking $F = \{f, f^{-1}\}$ as in Definition \ref{OB def}.  
\end{proof}

The second lemma is classical.  It states that a homeomorphism close to the identity can be factored or ``fragmented" into homeomorphisms supported on small balls.  

\begin{lemma} [\textit{The Fragmentation Lemma} \cite{EK}]  \label{fragmentation} 
Suppose that $\{B_1, \ldots, B_k\}$ is an open cover of a compact manifold $M$.  There exists a neighbourhood $V$ of the identity in $\Homeo_0(M)$ such that  any $g \in V$ can be factored as a composition 
$$
g = g_1 g_2 \cdots g_k
$$
with $\supp(g_i) \subseteq B_i$.
\end{lemma}

The proof of this Lemma is contained in the proof of Corollary 1.3 of  \cite{EK} of  R. D. Edwards and R. C. Kirby, it relies on the topological torus trick.  Note that, although the statement of Corollary 1.3 in \cite{EK} is not equivalent to our statement of the Fragmentation Lemma, our statement is precisely the first step in their proof.    

With these preliminaries, we can now prove Theorem \ref{Homeo OB}.

\begin{proof}[Proof of Theorem \ref{Homeo OB}]
Let $M$ be a closed manifold and $\{B_1, B_2, \ldots, B_k\}$  a cover of $M$ by embedded open balls.   Using Lemma \ref{fragmentation}, let $V$ be a neighbourhood of the identity in $\Homeo_0(M)$ such that any $g \in V$ can be factored as a product 
$$g = g_1 g_2 \cdots g_k$$
with $\supp(g_i) \subseteq B_i$.   We show that $V$ is relatively (OB).  To see this, let $U \subseteq V$ be an arbitrary symmetric identity neighborhood and fix some $\epsilon >0$ such that any homeomorphism $f$ with $\supp(f)$ of diameter $<\epsilon$ is contained in $U$.   For each $i = 1, \ldots, n$, choose a homeomorphism $h_i \in \Homeo_0(M)$ such that $h_i(B_i)$ is contained in a ball $B'_i$ of diameter $\epsilon$.  Then, for $g_i$ as above,  $\supp(h_i\inv g_i h_i) \subseteq B'_i$, so $h_i\inv g_i h_i \in U$ or equivalently $g_i \in h_i U h_i\inv$.   

Choose $p$ large enough so that $U^p$ contains all of the finitely many $h_i$.
Then 
$$g \in  h_1 U h_1\inv h_2 U ...   ... h_n U h_k\inv \subseteq U^{2kp+k},$$ showing that $V \subseteq U^{2kp+k}$ and thus that $V$ is relatively (OB). As $\Homeo_0(M)$ is connected, $V$ generates, so this proves (OB) generation.   
\end{proof}

\subsection{The fragmentation norm.}
Using Lemma \ref{fragmentation}, one can define the following word norm on $\Homeo_0(M)$.
\begin{defi}[Fragmentation norm]
Let $\mathcal{B} = \{B_1, B_2, \ldots, B_k \}$ be a cover of $M$ by embedded open balls.  The \emph{fragmentation norm} $\| \cdot \|_\mathcal{B}$ on $\Homeo_0(M)$ is given by 
$$\| g \|_\mathcal{B} = \min \{ n \mid g= g_1 g_2 \cdots g_n \text{ with }  \supp(g_i) \subseteq B_{k_i} \text{ for some } k_i\}$$
\end{defi}

Note that this is dependent on the choice of cover, and \emph{not} equivalent to the ``conjugation-invariant fragmentation norm" that appears elsewhere in the literature, e.g. in  \cite{BIP}.  

Let $d_\mathcal{B}$ denote the right-invariant metric induced by $\| \cdot \|_\mathcal{B}$.  The following observation is due to E. Militon \cite{Militon}.  

\begin{prop} 
Let $\mathcal{A}$ and $\mathcal{B}$ be finite covers of $M$ by embedded open balls.  Then $d_\mathcal{A}$ and  $d_\mathcal{B}$ are quasi-isometric.
\end{prop}

However, an even stronger statement is valid.  For any cover $\mathcal{B}$, the metric $d_\mathcal{B}$ is \emph{maximal}, in the sense given in the introduction -- i.e. it is quasi-isometric to that given by any relatively (OB) identity neighborhood $V$ in $\Homeo_0(M)$.  

\begin{thm}  \label{fragmentation norm thm}
Let $\mathcal{B} = \{B_1, \ldots,  B_k\}$ be a cover of $M$ by embedded open balls.  Then the identity map 
$$
(\Homeo_0(M), d_\mathcal{B}) \to \Homeo_0(M)
$$
is a quasi-isometry.  
\end{thm}

In particular, this proves Theorem \ref{loc OB} as stated in the introduction.

\begin{proof}
By Theorem \ref{Homeo OB}, we must show that the word metric $d_\mathcal B$ is quasi-isometric to some or all word metrics $d_V$ given by a relatively (OB) identity neighborhood $V$ in $\Homeo_0(M)$. 

So choose $V$ small enough so that each $g \in V$ can be fragmented as $g = g_1 \cdots g_k$ with $\supp(g_i) \subseteq B_i$.  Thus, if $f \in V^m$, then $\| f \|_\mathcal{B} \leqslant k m$ and so $d_\mathcal B\leqslant k\cdot d_V$.  For the reverse inequality, set $G_i =  \{g \in \Homeo_0(M) \mid \supp(g) \subseteq B_i\}$ and 
note that $V \cap G_i$ is a neighborhood of the identity in $G_i$, which is naturally isomorphic to a subgroup of $\Homeo_\partial(\B^n)$.  Lemma \ref{Bn lemma} now gives $n_i$ such that $G_i=(V\cap G_i)^{n_i} \subseteq V^{n_i}$.  Let $N = \max \{n_1, \ldots, n_k \}$.  Then $\|f \|_\mathcal{B} = m$ implies that $f \in V^{mN}$, i.e., $d_V\leqslant N\cdot d_{\mathcal B}$
\end{proof}
It follows from Theorem \ref{fragmentation norm thm} that the quasi-isometry type of $\Homeo_0(M)$ may be computed either from the word metric $d_V$ associated to a sufficiently small identity neighbourhood $V$ or from $d_\mathcal B$, where $\mathcal B$ is any covering by embedded open balls.

\subsection{Distortion in homeomorphism groups}

Having defined a maximal metric we can now define distortion for homeomorphism groups.  

\begin{defi}  \label{distortion def}
Let $H$ be an (OB) generated Polish group and $G$ an (OB) generated closed subgroup of $H$. We say that $G$ is \emph{undistorted} if the inclusion $G \to H$ is a quasi-isometric embedding and \emph{distorted} otherwise.  
\end{defi}

A related notion of distortion for cyclic subgroups of homeomorphisms appears in \cite{CF}, \cite{FH}, \cite{Militon2}, \cite{Polterovich} (see also references therein).    Without reference to a maximal metric on these groups, these authors define a homeomorphism (or diffeomorphism) $f$ to be \emph{distorted} if there exists a finitely generated subgroup $\Gamma$ of $\Homeo_0(M)$ (respectively, $\Diff_0(M)$)  containing $f$ in which the cyclic subgroup generated by $f$ is distorted in the sense above, i.e. if 
$$\lim \limits_{n \to \infty} \frac{ \| f^n \|_S}{n} = 0$$
where $\| f^n \|_S$ denotes the word length of $f^n$ with respect to some finite generating set $S$ for $\Gamma$.  
Though this definition appears somewhat artificial, it is useful: distortion has been shown to place strong constraints on the dynamics of a homeomorphism (see in particular \cite{FH}), and is used in \cite{Hurtado} to relate the algebraic and topological structure of $\Diff_0(M)$.

The following lemma, essentially due to Militon, relates our Definition \ref{distortion def} to the original notion of distorted elements.  
\begin{lemma}[Militon, \cite{Militon}]    
Let $M$ be a compact manifold, possibly with boundary, and let $\partial$ be a maximal metric on $\Homeo_0(M)$.   
If
$$\lim \limits_{n \to \infty} \frac{\partial(g^n, \id) \log(\partial(g^n, \id))}{n} = 0$$
then there exists a finitely generated subgroup $\Gamma \leqslant \Homeo_0(M)$ in which $g$ is distorted.  
\end{lemma}

\begin{proof} 
By Theorem \ref{fragmentation norm thm}, it suffices to take $\partial = d_{\mathcal B}$ to be the maximal metric given by the fragmentation norm with respect to some cover $\mathcal B$ of $M$.  Now fix $g \in \Homeo_0(M)$ and let $l_n = d_{\mathcal B}(g^n)$.  

In \cite[Prop 4.4]{Militon}, Militon uses the fragmentation property and a construction based on work of A. Avila to build a subsequence $\sigma(n)$ and a finite set $S \subseteq \Homeo_0(M)$ such that 
$$
\|g^{\sigma(n)}\|_S \leqslant C\,  l_{\sigma(n)}  \log \left( \sum_{i+1}^n l_{\sigma(i)} + C' \right)
$$ 
Moreover, using the fact that $\frac{l_n \log (l_n)}{n} \to 0$, the subsequence $\sigma(n)$ is chosen to increase rapidly enough so that 
$$ \frac{l_{\sigma(n)}  \log \left( \sum_{i+1}^n l_{\sigma(i)} + C' \right)}{\sigma(n)} \leqslant \frac{1}{n}$$
As the sequence $\|g^n\|_S$ is subadditive, this shows that $\lim \limits_{n \to \infty} \frac{\|g^n\|_S}{n} = 0$.  
\end{proof}


\section{Computation of quasi-isometry types}  \label{examples sec}
Whereas the results above allow us to talk unambiguously of the quasi-isometry type of a homeomorphism group, they provide little information as to what this might be or how to visualize it.    In this section, we give some examples.  The easiest case is that of spheres, where the homeomorphism group is simply quasi-isometric to a point. 
\begin{prop} \label{Sn OB prop}
$\Homeo_0(\S^n)$ has property (OB), hence a trivial quasi-isometry type.
\end{prop}
This was previously proved in \cite{OB}. Subsequently, in \cite{CF}, $\Homeo_0(\S^n)$ was shown to have property (OB) even viewed as a discrete group.  Both proofs use the annulus conjecture.  The proof we give here is similar in spirit, but only uses Kirby's torus trick.  

\begin{proof}[Proof of Proposition \ref{Sn OB prop}]
Let $x_1, x_2, x_3$ be distinct points in $\S^n$, and let $B_i = \S^n \setminus \{x_i\}$.  By Theorem \ref{fragmentation norm thm}, it suffices to show that the fragmentation norm on $\Homeo_0(\S^n)$ with respect to the cover $\{B_1, B_2, B_3\}$ is bounded. We do this in three short lemmas.
  
\begin{lemma} \label{lemma1}
Let $f \in \Homeo_0(\S^n)$, and $x\neq  y \in \S^n$.  Then $f$ can be written as a product $f_1 f_2 f_3$, where $f_1(x) = f_3(x) = x$, and $f_2(y) = y$.  
\end{lemma}

\begin{proof} If $f(x) = x$, set $f_2 = f_3 = \id$, and $f_1 = f$.  Otherwise, let $f_1$ be such that $f_1(x) = x$ and $f_1^{-1}(f(x)) = z$ for some $z \notin \{x, y\}$.  Let $f_2$ be such that $f_2(y) = y$ and $f_2^{-1}(z) =x$.  Then $f_3:= f_2^{-1}f_1^{-1}f$ satisfies $f_3(x) = x$, and $f = f_1 f_2 f_3$.  
\end{proof}

A nearly identical argument shows the following. We omit the proof.
\begin{lemma} \label{lemma2}
Let $f \in \Homeo_0(\S^n)$, $p\in \S^n$  and suppose $f(p) = p$. Let $x\neq y \in \S^n$ be distinct from $p$.  Then $f$ can be written as a product $f_1 f_2 f_3$, where $f_1(x) = f_3(x) = x$, $f_2(y) = y$, and each $f_i$ fixes $p$.  
\end{lemma}

Given distinct points $x_1, x_2, x_3 \in \Homeo_0(\S^n)$, Lemmas \ref{lemma1} and \ref{lemma2} together imply that any homeomorphism $f \in \Homeo_0(\S^n)$ can be written as a product of 9 homeomorphisms, each of which fixes two of the $x_i$.  We now use Kirby's torus trick to produce a further factorization.  

\begin{lemma} 
Let $f \in \Homeo_0(\S^n)$, let $x_1, x_2$ be distinct points in $\S^n$, and suppose $f(x_i) = x_i$.  Then $f = f_1 f_2$, where each $f_i$ fixes a neighbourhood of $x_i$.
\end{lemma}

\begin{proof}
Identify $\S^n \setminus \{x_2\}$ with $\R^n$.  Since $f$ fixes $x_1 \in \R^n$, there is a neighbourhood $N$ of $x_1$ such that $|f(x) - x| < \epsilon$ for all $x \in N$.   The topological torus trick produces a compactly supported homeomorphism $h$ of $\R^n$ such that $h(x) = f(x)$ for all $x$ in some compact set $K \subseteq N$ with nonempty interior.   Moving back to $\S^n$, $h$ induces a homeomorphism $f_2$ agreeing with $f$ on an open neighborhood of $x_1$ and fixing a neighbourhood of $x_2$.  Thus, $f_1:= f f_2^{-1}$ fixes a neighbourhood of $x_1$, and $f = f_1f_2$.  
\end{proof} 

It follows now that any $f \in \Homeo_0(\S^n)$ can be written as a product of at most 18 homeomorphisms, say $f_1, f_2, ...$ each of which is the identity on some neighborhood $N_i$ of a point $x_i$, hence supported in the ball $\S^n \setminus \{x_i\}$.    Fix a metric on $\S^n$ and let $V = V_\epsilon$.  Given $x_i \in \S^n$, let $g_i \in V$ be a homeomorphism such that $g_i (N_i)$ contains an $\epsilon$-radius ball about $x_i$.  Now $g_i f_i g_i^{-1}$ is supported on the ball $B_i \cong \B^n$ that is the complement of the $\epsilon$-radius ball about $x_i$.  Since, by Lemma \ref{Bn lemma}, $\Homeo_\partial(\B^n)$ has property (OB), $g_i f_i g_i^{-1} \in V^k$ for some fixed $k$, hence $f \in V^{18k}$.   Consequently, $\Homeo_0(\S^n)$ has property (OB).
\end{proof}

By contrast, with the exception of $M = \S^1$, as soon as $\pi_1(M)$ has an element of infinite order, the quasi-isometry type of $\Homeo_0(M)$ is highly nontrivial. In fact, we show that it contains an {\em isomorphic coarse embedding} of the Banach space $C([0,1])$,  i.e. an isomorphic embedding of the additive topological group $C([0,1])$, which is, moreover, a coarse embedding of $C([0,1])$ with the $\|\cdot\|_\infty$ metric. In many cases, we can even obtain an isomorphic quasi-isometric embedding of $C([0,1])$. Since every separable metric space isometrically embeds into $C([0,1])$, this shows that the quasi-isometry type of $\Homeo_0(M)$ is universal for all separable metric spaces in these cases.

For our calculations, the following basic lemma is needed.

\begin{lemma}[See also Lemma 3.10 in \cite{Militon}.] \label{militon lemma}
Let $M$ be a compact manifold with universal cover $\tilde M\overset{\pi}{\to}M$. Fix a compatible geodesic metric $d$ on $M$ and let  $\tilde{d}$ be the  geodesic lift to $\tilde{M}$. Fix also a maximal right-invariant metric $\partial$ on $\Homeo(M)$ and a compact set $D\subseteq \tilde M$.

Then there exist constants $K, C$ such that
$$
\partial(g, \id) \geqslant \frac 1K \sup_{x,y \in D} \tilde{d}(\tilde{g}(x),\tilde g(y)) - C
$$
holds for any $g \in \Homeo_0(M)$ and any lift $\tilde{g}\in \Homeo(\tilde{M})$ of $g$.  
\end{lemma}

Recall that, for any cover $\tilde{M} \overset{\pi} \to M$, a \emph{lift} of $g \in \Homeo_0(M)$  to $\tilde{M}$ is an element $\tilde{g}\in \Homeo(\tilde{M})$ such that $\pi \tilde{g} = g \pi$.   

\begin{proof}
Let $\mathcal{B}$ be a finite cover of $M$ by embedded open balls of diameter at most $1$.  Then for each $B\in \mathcal B$, the pre-image $\pi^{-1}(B)$ is a disjoint union of open sets of diameter $\leqslant 1$. 
Let $g \in \Homeo(\tilde{M})$ and suppose $\| g\|_\mathcal{B} = n$, so that $g = g_1\cdots g_n$ for some $g_i$ supported in $B_{i}\in \mathcal B$.  Let $\tilde{g}_i$ be the lift of $g_i$ supported in $\pi\inv(B_{i})$.  Then 
$\sup_{x\in \tilde M} \tilde{d}(\tilde{g}_i(x), x) \leqslant 1$, and repeated applications of the triangle inequality gives
$$\sup_{x\in \tilde M} \tilde{d}(\tilde{g}_1\tilde{g}_2 \cdots \tilde{g}_n(x), x) \leqslant n.$$

Then
\[\begin{split}
\sup_{x, y\in D} \tilde{d}(\tilde{g_1}\cdots \tilde{g_n}(x), \tilde{g_1}\cdots\tilde{g_n}(y))
&\leqslant 
2\sup_{x\in D} \tilde{d}(\tilde{g_1}\cdots \tilde{g_n}(x), x)+{\rm diam}_{\tilde d}(D)\\
&\leqslant 
2\sup_{x\in \tilde M} \tilde{d}(\tilde{g_1}\cdots \tilde{g_n}(x), x)+{\rm diam}_{\tilde d}(D)\\
&\leqslant 
2n +{\rm diam}_{\tilde d}(D) = 2\| g \|_\mathcal{B}+{\rm diam}_{\tilde d}(D). 
\end{split}\]

Now if $\tilde{g}$ is any lift of $g$, then $\tilde g=f\tilde{g_1}\cdots \tilde{g_n}$ for some deck transformation $f$.  As deck transformations act by isometries on $(\tilde M, \tilde{d})$, it follows that 
$$
\sup_{x, y\in D} \tilde{d}(\tilde g(x), \tilde{g}(y))=\sup_{x, y\in D} \tilde{d}(\tilde{g_1}\cdots \tilde{g_n}(x), \tilde{g_1}\cdots\tilde{g_n}(y))\leqslant 2\| g \|_\mathcal{B}+{\rm diam}_{\tilde d}(D).
$$
Since, by Theorem \ref{fragmentation norm thm} , $\|g\|_\mathcal B$ is affinely bounded in terms of $\partial(g,{\rm id})$, this proves the lemma.  
\end{proof}

\begin{prop} \label{C prop 1}
Let $M$ be a compact manifold of dimension at least 2 with infinite fundamental group.  Then there is a  isomorphic coarse embedding $C([0,1]) \to \Homeo_0(M)$ and thus $\Homeo_0(M)$ is coarsely universal among all separable metric spaces. 
\end{prop}
This result is a natural generalization of \cite[Ex. 6.8]{CF} of Calegari and Freedman, which provides a coarse isomorphic embedding of $\Z$.   We obtain an even stronger result if more is known about the structure of the fundamental group is $M$. 

\begin{prop} \label{C prop 2}
Let $M$ be a compact manifold of dimension at least 2, and suppose that $\pi_1(M)$ contains an undistorted infinite cyclic subgroup.  Then there is a  isomorphic quasi-isometric embedding $C([0,1]) \to \Homeo_0(M)$.  Moreover, the image of this embedding can be taken to have support in any neighborhood of a curve representing a generator for the undistorted subgroup.  
\end{prop}

\begin{proof} [Proof of Propositions \ref{C prop 1} and \ref{C prop 2}]
Let $C_b([0,1])=\{f\in C([0,1])\del f(0)=0\}$ and observe that $C_b([0,1])$ is a closed hyperplane in $C([0,1])$, i.e., a closed linear subspace of codimension $1$. By Proposition 4.4.1 of \cite{kalton}, $C([0,1])$ is linearly isomorphic to its hyper-planes and thus to $C_b([0,1])$. It therefore suffices  to produce an isomorphic coarse, respectively quasi-isometric,  embedding $C_b([0,1])\to \Homeo_0(M)$. 

Let $\tilde{M} \overset{\pi}\to M$ be the universal cover and fix a compatible geodesic metric $d$ on $M$.  Denote the geodesic lift of $d$ to $\tilde{M}$ by $\tilde d$ and choose a compact fundamental domain $D\subseteq \tilde M$ and basepoint $z\in D$.

We begin with the case where $\pi_1(M)$ contains an undistorted infinite cyclic subgroup.  If $\dim(M) = 2$, then the classification of surfaces implies that we may find a simple closed curve $\gamma$ representing a generator of some such subgroup.  If $\dim(M) > 2$, any curve $\gamma$ representing an infinite order element may be perturbed so as to be embedded.  
Thus, in either case, we have an embedded curve $\gamma \subseteq M$ representing an infinite order, undistorted element of $\pi_1(M)$.  

Let $\alpha\colon \B^{n-1} \times \S^1 \to M$ be an embedding with $\alpha(\{0\} \times \S^1) = \gamma$
and define a homomorphism $\phi\colon C_b([0,1]) \to \Homeo_0(M)$ by 
$$
\phi(f)(x)=\begin{cases}
x &\text{ if } x\notin {\rm im}\,\alpha, \\
\alpha\big(b, \, e^{i(\theta + f(1-\norm{b}))}\big) &\text{ if } x = \alpha(b, \, e^{i \theta}).
\end{cases}
$$
It is easy to check that $\phi$ is a well defined isomorphic embedding of the topological group $C_b([0,1])$ into $\Homeo_0(M)$ 
and so, in particular, is \emph{Lipschitz for large distances} with respect to any maximal metric $\partial$ on $\Homeo_0(M)$, meaning that 
$$\partial (\phi(g), \phi(f)) \leqslant K \norm{f-g}_\infty + C$$ 
for some constants $K$ and $C$.  We will show now that it is also a quasi-isometric embedding.  

Let $f \in C_b([0,1])$ and let $n=\big\lfloor \norm{f}_\infty\big\rfloor$.  Let $\tilde{\phi}(f)$ denote the lift of $\phi(f)$ that fixes points outside of $\pi^{-1}(\alpha(\B^{n-1} \times \S^1))$.  
By the construction of $\phi(f)$, there exists $x \in D$ such that 
$$
\tilde d\big(\tilde\phi(f)(x), \gamma^{n}(x)\big)\leqslant {\rm diam}_{\tilde d}(D).
$$
In particular, as $\tilde\phi(f)$ also fixes a point $y$ of $D$, we have 
\[\begin{split}
\sup_{z,v\in D}\tilde d(\tilde\phi(f)(z),\tilde\phi(f)(v))
&\geqslant  \tilde d(\tilde\phi(f)(x),\tilde\phi(f)(y))\\
&\geqslant  \tilde d(\tilde\phi(f)(x),y)\\
&\geqslant \tilde d(\gamma^n(x),x)-\tilde d(\tilde\phi(f)(x),\gamma^n(x))-\tilde d(x,y)\\
&\geqslant \tilde d(\gamma^n(x),x)-2{\rm diam}_{\tilde d}(D).
\end{split}\]
So, as $n\mapsto \gamma^n(x)$ is a quasi-isometric embedding of $\N$ into $(\tilde M,\tilde d)$ with constants independent of $x$, using Lemma \ref{militon lemma} we see that also $f\mapsto \phi(f)$ is a quasi-isometric embedding.

Note that if instead $\gamma$ generates a distorted but coarsely embedded copy of $\Z$ inside $\pi_1(M)$, this construction gives a coarse embedding $C_b([0,1] )\to \Homeo_0(M)$.  

For the more general case where $\pi_1(M)$ is infinite, but $M$ is not known even to have an undistorted infinite order element (whether this case can occur appears to be an open question), we may still produce a quasi-isometric embedding of $\R$ in $\tilde{M}$ when $\tilde{M}$ is endowed with a lifted geodesic metric $\tilde d$.  To see this, take geodesic segments $\gamma_n$ of length $2n$ centered at $\id$ in the Cayley graph of $\pi_1(M)$;  by K\"onig's lemma, the union of all paths $\gamma_n$ has an infinite branch, which gives an isometrically embedded copy of the metric space $\Z$ in $\pi_1(M)$, hence a quasi-isometric homeomorphic embedding of $\R$ into $(\tilde{M}, \tilde d)$.
 
Provided that $\dim(M) > 2$, we may perturb the embedding of $\R$ in $\tilde{M}$ so that it remains quasi-isometric, while its projection to $M$ under the covering map $\pi\colon \tilde{M} \to M$ is also homeomorphically embedded.  In fact, by moving each point no more than $\epsilon$, we can arrange so that the image of $[-n, n]$ under projection to $M$ has a neighborhood of diameter $\epsilon^{2n}$ that avoids the rest of the curve.   Call this embedded infinite curve $\hat{\gamma} \subseteq M$, it will play the role of $\gamma$ in the previous construction.  Our choice of $\hat{\gamma}$ means that there is a homeomorphic embedding $\alpha\colon \B^{n-1} \times \R \to M$ with $\alpha(\{0\} \times \R) = \hat{\gamma}$.

Modifying the previous construction, we can define an isomorphic embedding $\phi$ of the topological group  $(C_b([0,1]),+)$ into $\Homeo_0(M)$, supported in the image of $\alpha$, that translates points along the curve $\hat\gamma$.   For this, let $g_t$ be a transitive $1$-parameter flow on $\R$ such that $|g_t(x) - x| \to 0$ as $|x| \to \infty$.  Such a flow may be obtained by conjugating the standard translation flow on $\R$ by a suitable homeomorphism, for example, a homeomorphism that behaves like $x \mapsto e^x$ near $+ \infty$ and $x \mapsto -e^{-x}$ near $-\infty$.   Note that, by appropriate choice of conjugacy, we can have $|g_t(x) - x|$ approach 0 as slowly as we like.  

Now, given $f \in C_b([0,1])$, define 

$$
\phi(f)(x)=\begin{cases}
x &\text{ if } x \notin{\rm im} \,\alpha, \\
\alpha\big(b, \, g_{f(1-\norm{b})}(t) \big) &\text{ if } x = \alpha(b, \, t).
\end{cases}
$$

As before, $\phi$ gives a well defined isomorphic embedding of the topological group $C_b([0,1])$ into $\Homeo_0(M)$.  The role of the flow $g_t$ is to ensure that this actually defines a {\em homeomorphism} $\phi(f)$ of $M$ and not only a permutation.

Now, if $f_n\in C_b([0,1])$ is a sequence with $\norm{f_n}_\infty\to \infty$, then, as before, we find $\tau_n\in \pi_1(M)$ so that $\tilde d\big(\phi(f_n)[D],  \tau_n[D]\big)\leqslant {\rm diam}_{\tilde d}(D)$ and $\tau_n\to \infty$ in $\pi_1(M)$. Repeating the argument above, we thus see that $\{\phi(f_n)\}_n$ cannot be relatively (OB) in $\Homeo_0(M)$ and thus that $\phi$ is a coarse embedding.
\end{proof}

\begin{rem}
The construction in the proof above also naturally gives a quasi-isometric embedding $\pi_1(M) \to \Homeo_0(M)$.   Here is a brief sketch of the required modification.  Take a finite generating set for $\pi_1(M)$ and choose a loop $\gamma_i \subseteq M$ representing each generator -- by abuse of notation, denote the generator also by $\gamma_i$.  
For each $i$, let $\alpha_i\colon \B^{n-1} \times \S^1 \to M$ be an embedding with image in a small neighborhood of $\gamma_i$, and use the construction in the proof of Proposition \ref{C prop 2} to produce quasi-isometric embeddings $\phi_i\colon C_b([0,1]) \to \Homeo_0(M)$ with $\phi_i\big[C_b([0,1])\big]$ consisting of homeomorphisms supported in $\alpha_i(\B^{n-1} \times \S^1)$.   Let $f \in C_b([0,1])$ be the constant function $f(x) = x$.   The reader familiar with the ``point-pushing" construction in the study of based mapping class groups will recognize $\phi_i(f)$ as a point-push around $\gamma_i$.  For each element $\gamma$ of $\pi_1(M)$, choose an expression $\gamma = \gamma_{i_1} \cdots \gamma_{i_k}$ of minimal length, and set $\Phi(\gamma) = \phi_{i_1}(f) \cdots \phi_{i_k}(f)$.  One then checks that this gives a quasi-isometric embedding, though not a homomorphism, $\Phi\colon \pi_1(M) \to \Homeo_0(M)$.    Whether such a quasi-isometric embedding can be constructed so as to also be a homomorphism appears to be an interesting question, especially for the case of compact 2-manifolds of higher genus.   
\end{rem}

In the special case of $2$-manifolds, Militon also proved the reverse of the inequality given in Lemma \ref{militon lemma}.   

\begin{thm}[Militon, Theorem 2.9 in \cite{Militon}] \label{militon thm}
Let $M$ be a compact $2$-dimensional manifold, possibly with boundary, with universal cover $\tilde M\overset{\pi}\to M$. Fix a geodesic metric $d$ on $M$ with geodesic lift $\tilde d$, $\mathcal{B}$ a finite cover of $M$ by embedded open balls and $D\subseteq \tilde M$ a compact subset with $\pi[D]=M$.  Then there exist $K$ and $C$ such that 
$$
\frac 1K \sup_{x,y \in D} \tilde{d}(\tilde{g}(x), \tilde g(y)) - C \leqslant d_\mathcal B(g,1) \leqslant K \sup_{x,y \in D} \tilde{d}(\tilde{g}(x), \tilde g(y)) + C$$
holds for all $g \in \Homeo_0(M)$ and lifts $\tilde{g}\in \Homeo(\tilde{M})$.  
\end{thm}
Militon's proof relies both on 2-dimensional plane topology, as well as on the particular algebraic structure of fundamental groups of surfaces, so the following question remains open. 

\begin{quest} Does Theorem \ref{militon thm} hold for $n$-manifolds, when $n \geqslant 3$?
\end{quest}

As a more basic step, one can also ask the following.  

\begin{quest}
Are there examples of manifolds with finite or trivial fundamental group that fail to have property (OB)?   
\end{quest}


\section{Covers, extensions, and bounded cohomology} \label{covers sec}

In this section, we relate the large scale geometry of the group of homeomorphisms of a manifold $M$ to the group of homeomorphisms of its universal cover that are equivariant with respect to the action of $\pi_1(M)$.  In other words, this is the group of \emph{lifts} of homeomorphisms of $M$ to $\tilde{M}$.   A basic example to keep in mind is $M = \T^n$.  In this case, the group of lifts of homeomorphisms to the universal cover can be identified with $\Homeo_{\Z^n}(\R^n)$, the orientation preserving homeomorphisms of $\R^n$ that commute with integral translations.  

The main theorem of this section, Theorem \ref{A section}, gives a condition for the group of lifts to be quasi-isometric to the direct product $\Homeo_0(M) \times \pi_1(M)$.   As one application, we show that ${\rm Homeo}_{\Z^n}(\R^n)$ is quasi-isometric to ${\rm Homeo}_0(\T^n)\times \Z^n$, and in the course of the proof produce a bornologous section of the natural map $\Homeo_{\Z^n}(\R^n)\overset{\pi}\longrightarrow \Homeo_0(\T^n)$, that is, a section of $\pi$ that is simultaneously a coarse embedding of $\Homeo_0(\T^n)$ into $\Homeo_\Z(\R)$.  
This result is particularly interesting in the context of bounded cohomology.  A theorem of A. M. Gersten (see Theorem \ref{gersten thm} below) states that, if $H$ is a \emph{finitely generated} group and a central extension 
$$0 \to \Z^n \to G \overset{\pi}\to H \to 1$$
represents a bounded class in the second cohomology of $G$, then there exists a quasi-isometry $\phi: H \times \Z^n \to G$ such that, for all $(h, x) \in H \times \Z^n$, the image $\pi \phi(h, x)$ is a uniformly bounded distance from $h$ (with respect to any fixed word metric on $H$).  The converse to this theorem appears to be open.   
Here, we give a new proof of Gersten's theorem that is applicable in a broader context, and then show that a variation on the converse is false, namely, that the central extension
$$
0\til \Z^n\to  \Homeo_{\Z^n}(\R^n)\overset{\pi}\longrightarrow \Homeo_0(\T^n) \to 1
$$
does not represent a bounded cohomology class in $H^2({\rm Homeo}_0(\T^n), \Z^n)$.

\subsection{Extensions via universal covers} \label{extension sec}
We begin with a general lemma on bornologous sections of central extensions, generalizing well-known facts from the case of discrete groups. To fix terminology, say that a subgroup $K$ of a topological group $G$ is {\em coarsely embedded} if the inclusion map $K\inj G$ is a coarse embedding.

\begin{prop}\label{commuting subgroups}
Suppose $G$ is a topological group generated by two commuting subgroups $K$ and $H$, i.e., so that $[K,H]=1$. Assume also that $K$ is coarsely embedded and that the quotient map $\pi\colon G\til G/K$ admits a bornologous section  $\phi\colon G/K\til H$. Then $K\times G/K$ is coarsely equivalent to $G$ via the map
$$
(x,f)\mapsto x\phi(f).
$$
\end{prop}

\begin{proof}For simplicity of notation, set $\tilde g=\phi\pi(g)$. As $\phi$ and $\pi$ are both bornologous, so is the composition $g\in G\mapsto \tilde g\in H$.

The first step is to show that every relatively (OB) subset $A\subseteq G$ is contained in some product $BC$, where $B\subseteq H$ and $C\subseteq K$  are relatively (OB) in $H$ and $K$ respectively.   To see this, let $A\subseteq G$ be a relatively (OB) subset and let $B=\tilde A$, which is relatively (OB) in $H$.  Let $C = (B\inv A)\cap K$, which is relatively (OB) in $G$ and thus also in $K$, since $K$ is coarsely embedded.   If $a\in A$, then $\tilde a\inv a\in (B\inv A)\cap K=C$, so $a=\tilde a\cdot \tilde a\inv a \in B\cdot C$.  Thus, $A\subseteq BC$ as required.

We now claim that $g\in G\mapsto\tilde g\inv g\in K$ is bornologous. To see this, suppose that $A\subseteq G$ is relatively (OB) and find $B\subseteq H$ and $C\subseteq K$ as above. Fix first a relatively (OB) set $D\subseteq H$ so that
$$
fg\inv \in A\saa \tilde f\tilde g\inv \in D.
$$
Now, suppose $g,f\in G$ are  given  and set $x=\tilde g\inv g\in K$ and $y=\tilde  f\inv f\in K$. Assume that
$$
 \tilde f\tilde g\inv\cdot yx\inv = fg\inv\in A\subseteq BC.
$$
Since $\tilde f\tilde g\inv \in H$, $B\subseteq H$, $ yx\inv \in K$ and  $C\subseteq K$, there is some $a\in H\cap K$ with $ \tilde f\tilde g\inv a\inv \in B$ and $ayx\inv \in C$. Therefore
$$
a\inv \in \tilde g\tilde f\inv B\subseteq D\inv B
$$
and thus also
$$
 yx\inv\in a\inv C \subseteq D\inv BC\cap K.
$$
In other words, 
$$
 fg\inv\in A\saa  \big(\tilde f\inv f\big)\cdot\big(\tilde g\inv g\big)\inv=yx\inv \in D\inv BC\cap K.
$$
As $K$ is coarsely embedded, $D\inv BC\cap K$ is relatively (OB) in $K$, so this shows that the map $g\mapsto \tilde g\inv g$ is bornologous as claimed.

It thus follows that also $g\in G\mapsto (\tilde g\inv g,\pi(g))\in K\times G/K$ is bornologous with inverse
$$
(x,f)\in K\times G/K\mapsto x\phi(f)\in G.
$$
Since the inverse is a composition of the bornologous map $(x,f)\in K\times G/K\mapsto (x,\phi(f))\in K\times H$ and the continuous homomorphism $(x,h)\in K\times H\mapsto xh\in G$, it is also bornologous, whence $g \mapsto (\tilde g\inv g,\pi(g))$ is a coarse embedding of $G$ into $K\times G/K$.
Moreover, being surjective, we see that it is a coarse equivalence.
\end{proof}

\begin{exa}
An intended application of Proposition \ref{commuting subgroups} is when $G$ is a topological group, $K\triangleleft G$ is a discrete normal subgroup and $H\leqslant G$ is a subgroup without proper open subgroups. Then the conjugation action of $H$ on $K$ is continuous and so, as $H$ cannot act non-trivially and continuously on a discrete set, we find that $[K,H]=1$. 
\end{exa}

\begin{rem}
The assumption that $K$ is coarsely embedded in the statement of Proposition \ref{commuting subgroups} is not superfluous.  Whereas a closed subgroup $K$ of a locally compact group is automatically coarsely embedded, this may not be the case for general Polish groups. For an extreme counter-example, consider the homeomorphism group of the Hilbert cube, ${\rm Homeo}([0,1]^\N)$.  This is a universal Polish group, i.e., contains every Polish group as a closed subgroup \cite{uspenskii}. Nevertheless it has property (OB) and is therefore quasi-isometric to a point \cite{OB}. Thus, for example, $\R$ can be viewed as a closed subgroup of ${\rm Homeo}([0,1]^\N)$, but is evidently not coarsely embedded.
\end{rem}

One must exercise a bit of care when constructing bornologous sections. Indeed, as the next example shows, it is not in general sufficient that the section be length preserving. 

\begin{exa}
Let $X$ be a Banach space and $Y$ a closed linear subspace with quotient map $\pi\colon X\til X/Y$. By the existence of Bartle--Graves selectors (Corollary 7.56 \cite{fabian}), 
there is a continuous function $\phi\colon X/Y\til X$ lifting $\pi$ so that
$\norm{\phi\pi(x)}\leqslant 2\norm{\pi(x)}$ for all $x\in X$. Thus, for all $x\in X$, 
$$
\norm{x}\leqslant \norm{\phi\pi(x)}+\norm{x-\phi\pi(x)}\leqslant 2\norm{\pi(x)}+\norm{x-\phi\pi(x)},
$$
while 
$$
\norm{\pi(x)}+\norm{x-\phi\pi(x)}\leqslant \norm x+\norm{x}+\norm{\phi\pi(x)}\leqslant 4\norm x.
$$
Combining the two estimates, we find that $x\mapsto \big(\pi(x), x-\phi\pi(x)\big)$ is a bijection between $X$ and $X/Y\oplus Y$ so that 
$$
\frac 12 \norm x\leqslant \norm{\pi(x)}+\norm{x-\phi\pi(x)}\leqslant 4\norm x.
$$
Nevertheless, this map need not be a coarse equivalence between $X$ and $X/Y\oplus Y$.
For example, $\ell^1$ is a surjectively universal separable Banach space, meaning that, for every separable Banach space $Z$ there is a closed linear subspace $Y\subseteq \ell^1$ so that $\ell^1/Y\iso Z$.  So find some closed linear subspace $Y\subseteq \ell^1$ so that $\ell^1/Y\iso c_0$. Then $\ell^1$ is not coarsely equivalent to $c_0\oplus Y$, since $c_0$ does not even coarsely embed into $\ell^1$.
\end{exa}


We now specialize to the case of homeomorphism groups of manifolds and lifts to covers.  Fix a compact manifold $M$ with universal cover $p\colon \tilde M\to M$, and let $G \leqslant \Homeo(\tilde{M})$ denote the group of lifts of homeomorphisms of $M$ to the universal cover $\tilde{M}$.   There is a natural projection $\pi\colon G \to \Homeo_0(M)$ defined by $\big(\pi g \big) (px) = p \big(g(x)\big)
$
for any $x\in \tilde M$ and $g\in G$.  This 
gives a short exact sequence
$$
\pi_1(M) \to G  \overset{\pi}\longrightarrow \Homeo_0(M).
$$
where $\pi_1(M) \leqslant \Homeo(\tilde{M})$ is the group of deck transformations.   

We now describe a subgroup of $G$ that will be used to apply Proposition \ref{commuting subgroups}, this is the group of lifts of paths.  If $f_t$ is a path in $\Homeo_0(M)$ from $f_0 = \id_M$ to $f_1$, then there is a unique lift $\tilde{f}_t$ of the path in $G$ with $\tilde{f}_0 = \id_{\tilde M}$, and $ \pi \tilde{f}_t = f_t$ for all $t$.    (Uniqueness follows from the fact that any two lifts of $\tilde{f}_t$ differ by a deck transformation, and the deck group is discrete.)  
Concatenating paths shows that the set 
$$
H = \{ \tilde{f_1} \in G\mid \tilde{f_t} \text{ is a lifted path with } \tilde{f}_0 =\id_{\tilde M} \}
$$
is in fact a subgroup of $G$.  As $H$ is path-connected and $\pi_1(M)$ is a discrete normal subgroup of $G$, it follows that $H$ and $\pi_1(M)$ commute, i.e., $[\pi_1(M),H]=1$.   Finally, $H$ and $\pi_1(M)$ generate $G$, since for any $g \in G$, we can let $\tilde{g} \in H$ denote the endpoint of a lift of a path from $\id$ to $\pi(g)$, in which case $\tilde{g} g^{-1} \in \pi_1(M)$, and $g \in H\pi_1(M)$. 

Thus, by Proposition \ref{commuting subgroups}, if we can show that the quotient map
$$
\pi\colon G\til {\rm Homeo}_0(M)
$$
admits a bornologous section $\phi\colon {\rm Homeo}_0(M)\til H$, and that $\pi_1(M)$ is coarsely embedded in $G$, then $G$ will be coarsely equivalent to $\Homeo_0(M) \times \pi_1(M)$.   Our next goal is to rephrase this in terms of a condition on the fundamental group $\pi_1(M)$.

Let $A = \pi_1(M) \cap H$ be the group consisting of lifts of loops in $\Homeo_0(M)$.  As $A$ is central in $\pi_1(M)$, it can be considered a subgroup of the abelianization $H_1(M; \Z)$, and so is finitely generated.   This gives a central extension
$$
A \to H \overset{\pi}\to \Homeo_0(M).
$$
The content of the following theorem is to reduce the problem of finding a bornologous section $\Homeo_0(M) \to H$ to finding a bornologous section of $\pi_1(M)/A \to \pi_1(M)$.

\begin{thm}\label{A section}
Assume that the quotient map $\pi_1(M) \to \pi_1(M)/A$ admits a bornologous section $\psi\colon \pi_1(M)/A \to \pi_1(M)$.  Then there is a bornologous section 
$$
\phi\colon {\rm Homeo}_0(M)\til H
$$ 
of $\pi$, whence the group $G$ of all lifts is quasi-isometric to $\Homeo_0(M) \times \pi_1(M)$.  
\end{thm}

\begin{proof}
We first show that $\pi_1(M)$ is coarsely embedded in $G$ and then construct a bornologous section $\phi\colon {\rm Homeo}_0(M)\til H$. 
Fix a proper metric $d$ on $\tilde M$ invariant under the action of $\pi_1(M)$ by deck transformations and let $D\subseteq \tilde M$ be a relatively compact fundamental domain for this action.

Then, since $H$ commutes with $\pi_1(M)$ and $\pi_1(M)$ acts by isometries, we see that, for $g,f\in H$, 
$$
\sup_{x\in \tilde M}d(g(x),f(x))= \sup_{x\in D}\sup_{\gamma\in \pi_1(M)}d(g\gamma(x),f\gamma(x))=\sup_{x\in D}d(g(x),f(x))<\infty,
$$
so $d_\infty(g,f):=\sup_{x\in \tilde M}d(g(x),f(x))$ defines a compatible right-invariant metric on $H$.

Fix also a point $x_0\in D$ and note that, since $\pi_1(M)$ acts properly discontinuously on $\tilde M$ by isometries,  
$$
d_{\pi_1(M)}\big(\gamma, \tau)=d(\gamma\inv (x_0),\tau\inv(x_0)\big)
$$ 
defines a proper right-invariant metric on $\pi_1(M)$, i.e. a metric whose bounded sets are finite. 

We now define a right-invariant metric on $G$ by setting
$$
\partial(g\gamma,f\tau)=\inf_{a\in A}\;\;d_\infty(g,fa)+d_{\pi_1(M)}(\gamma, \tau a\inv)
$$
for all $g,f\in H$ and $\gamma, \tau\in \pi_1(M)$. To see that $\partial$ is a \emph{compatible} metric, note that $G$ is isomorphic to the quotient of $H\times \pi_1(M)$ by the central subgroup 
$$
N=\{(a,a\inv)\in H\times \pi_1(M)\del a\in A\}
$$ 
and $\partial$ is simply the Hausdorff metric on the quotient $\frac{H\times \pi_1(M)}N$ induced by the metric $d_\infty+d_{\pi_1(M)}$ on $H\times \pi_1(M)$.

To verify that $\pi_1(M)$ is coarsely embedded in $G$, we need to show that, if  $\gamma_n\in \pi_1(M)$ is a sequence with $\partial(\gamma_n,1)$ bounded, then the sequence $d_{\pi_1(M)}(\gamma_n, 1)$ is also bounded. Indeed, if $\partial(\gamma_n,1)$ is bounded, then by definition of $\partial$, there are $a_n\in A$ so that both $d_\infty(1,a_n)$ and $d_{\pi_1(M)}(\gamma_n, a_n\inv)$ are bounded, whereby, as
$$
d_{\pi_1(M)}(1,a_n\inv)=d(x_0,a_n(x_0))\leqslant d_\infty(1,a_n),
$$ 
the sequence $d_{\pi_1(M)}(\gamma_n, 1)$ is also bounded.

Let now $S=\psi(\pi_1(M)/A)$, which is a transversal for $A$ in $\pi_1(M)$.  By Proposition \ref{commuting subgroups} (or Gersten's original theorem) applied to $A \to \pi_1(M) \to \pi_1(M)/A$, we conclude that $A\times \pi_1(M)/A$ is coarsely equivalent with $\pi_1(M)$ via the map $(a,k)\mapsto a\psi(k)$. In particular, writing every $\gamma\in \pi_1(M)$ in its unique form $\gamma=as$ with $a\in A$ and $s\in S$, we see that the map $\gamma=as\in \pi_1(M)\mapsto a\in A$ is bornologous. Therefore, for every constant $C$, there is a finite set $F\subseteq A$ so that, for all $a,b\in A$ and $s,t\in S$,
$$
d_{\pi_1(M)}(as,bt)\leqslant C\;\;\;\Rightarrow\;\;\;  ab\inv \in F.
$$

Since $S\inv$ is also a transversal for $A$ in $\pi_1(M)$, the set $S\inv D$ is a fundamental domain for the action  $A\curvearrowright \tilde M$. In particular, every $g\in \Homeo_0(M)$ has a unique lift $\phi(g)\in H$ so that $\phi(g)x_0\in S\inv D$.   This is our definition of the section $\phi$.

We now show that $\phi$ is bornologous. For this, suppose that $B\subseteq \Homeo_0(M)$ is relatively (OB).  Define $B^\phi=\{\phi(g)\phi(f)\inv \in H\del gf\inv \in B\}$.  We must show that $B^\phi$ is relatively (OB) in $H$.
To do this, let $V$ be an arbitrary identity neighbourhood in $H$. Without loss of generality, we may suppose that 
$$
V=\{g\in H\del d_\infty(g,1)<\eps\}
$$
for some $\eps>0$. Then $\pi[V]$ is an identity neighbourhood in $\Homeo_0(M)$ and thus, as $\Homeo_0(M)$ is connected, there is $m$ so that $B\subseteq \pi[V]^m=\pi[V^m]$. Let $F\subseteq A$ be the finite subset chosen as a function of $C=m\eps+2{\rm diam}(D)$ as above. 

Assume that $gf\inv \in B$. Since $\phi(g)\phi(f)\inv$ is a lift of $gf\inv$, there is some $a\in A$ so that $a\phi(g)\phi(f)\inv\in V^m$ and so 
$$
d(a\phi(g)x_0, \phi(f)x_0)\leqslant d_\infty(a\phi(g), \phi(f))< m\eps.
$$
Writing $\phi(g)x_0=s\inv x$ and $\phi(f)x_0=t\inv y$ for some $x,y\in D$ and $s,t\in S$, we have
\[\begin{split}
d_{\pi_1(M)}(a\inv s,t)
&= d(as\inv x_0,t\inv x_0)\\
&\leqslant  d(as\inv x_0,as\inv x)+d(as\inv x, t\inv y)+d(t\inv y, t\inv x_0)\\
&=  d(x_0,x)+d(a\phi(g)x_0,\phi(f)x_0)+d(y, x_0)\\
&<m\eps+2{\rm diam}(D),
\end{split}\]
whence $a\inv \in F$ and hence $\phi(g)\phi(f)\inv\in FV^m$. So $B^\phi\subseteq FV^m$, showing that $B^\phi$ is relatively (OB) in $H$.

Proposition \ref{commuting subgroups} now implies that $G$ is coarsely equivalent to $\Homeo_0(M) \times \pi_1(M)$. As $\Homeo_0(M)$ and $\pi_1(M)$ are both (OB) generated, so is $G$ and the coarse equivalence is in fact a quasi-isometry.
\end{proof}

The following examples illustrate two extreme cases in which Theorem \ref{A section} immediately applies.  

\begin{exa}[Case where $A = \pi_1(M)$]  \label{torus example}
Let $n \geqslant 1$, and let $\Homeo_{\Z^n}(\R^n)$ be the group of orientation preserving homeomorphisms of $\R^n$ that commute with integral translations.  As each deck transformation can be realized as a lift of a path of rotations of the torus, we have  $A = \pi_1(M)$ and so $\pi_1(M)/A$ is trivial.  Thus, Theorem \ref{A section} implies that
$\Homeo_{\Z^n}(\R^n)$ is quasi-isometric to $\Homeo_0(\T^n) \times \Z^n$.  
\end{exa}

Since $\Homeo_0(\T)$ has property (OB) (by \cite{OB}, or Proposition \ref{Sn OB prop}) and thus is quasi-isometric to a point, we immediately have the following corollary.
\begin{cor} \label{Z cor}
$\Homeo_{\Z}(\R)$ is quasi-isometric to $\Z$.
\end{cor}

\begin{exa}[Case where $A$ is trivial] \label{hyperbolic ex}
Let $M$ be a compact manifold admitting a metric of strict negative curvature.  Then $\pi_1(M)$ is centerless, so $A = 1$ and $\Homeo_0(M)$ is isomorphic to the identity component $G_0$ of $G$. Thus, $\Homeo_0(M) \times \pi_1(M)$ is isomorphic to $G_0 \times \pi_1(M) = G$.  
One can also see this directly as follows.  That $\tilde{M}$ has curvature bounded by $k<0$ implies that
$$ \sup \limits_{x \in \tilde{M}} \tilde{d}(x, \gamma x) = \infty$$
for any $1\neq \gamma \in \pi_1(M)$.  
But if $f_t$ is a path in $\Homeo_0(M)$ with lift $\tilde{f}_t$ based at $\id$ in $G$, then $\tilde d(\tilde{f}_1(x), x) < \infty$, so $f \mapsto \tilde{f}_1$ gives a well-defined, continuous, surjective lift $\Homeo_0(M) \to G_0$. 
\end{exa}


\subsection{Bounded cohomology.}
We briefly discuss the relationship between bounded cohomology, quasimorphisms, and central extensions.  
Recall that, if C is a group, each central extension $0 \to A \to B \to C \to 1$ defines an element of the (discrete) group cohomology $H^2(C; A)$, this class is \emph{bounded} if it admits a cocycle representative $C^2 \to A$ with bounded image.  

It is well known that the extension
$$
0 \to \Z \to \Homeo_{\Z}(\R) \to \Homeo_0(\T) \to 1
$$
represents a bounded second cohomology class; this is a direct consequence of the Milnor-Wood inequality.   

We will show, by contrast, that 
$$
0 \to \Z^n \to \Homeo_{\Z^n}(\R^n) \to \Homeo_0(\T^n) \to 1
$$
does not represent a bounded element of $H^2\big( \Homeo_0(\T^n); \Z^n\big)$ for any $n>1$.   Although this can be demonstrated directly in the language of group cohomology, we choose to take this opportunity to introduce the notion of \emph{quasimorphism} in the general context of topological groups.  

The standard definition of a quasimorphism on a group $H$ is a function $\phi\colon H \to \R$ such that $\{ \phi(a) + \phi(b) - \phi(ab) \mid a,b \in H \}$ is a bounded set.  This generalizes quite naturally to functions  $\phi\colon H\til G$ with image in an arbitrary topological group $G$. For this, say that  two maps $\phi, \psi\colon X\til (Y,\ku E)$ from a set $X$ with values in a coarse space $(Y,\ku E)$ are  {\em close} if there is some entourage $E\in \ku E$ so that $(\phi(x),\psi(x))\in E$ for all $x\in X$.
\begin{defi}
A map $\phi\colon H\til G$ from a group $H$ to a topological group $G$ is a {\em quasimorphism} if  the two maps $(a,b)\mapsto \phi(ab)$ and $(a,b)\mapsto \phi(a)\phi(b)$ are close with respect to the coarse structure on $G$. 
\end{defi}

When working with the right-invariant coarse structure on $G$, this simply means that the symmetric set  
$$
D=\{  \phi(a) \phi(b)\phi(ab)^{-1} \mid a,b \in H \}^\pm\cup\{1\}
$$
is relatively (OB) in $G$. Observe that then $\phi(1)=\phi(1)\phi(1)\phi(1)\inv\in D$ and so $\phi(a\inv)\in \phi(a)\inv D\phi(aa\inv)=\phi(a)\inv D\phi(1)\subseteq \phi(a)\inv D^2$ for all $a\in H$. Therefore, if $D$ is contained in the center $Z(G)$ of $G$, we find that, for all $h_1,\ldots, h_k\in H$ and $\alpha_1, \ldots, \alpha_k\in \{-1,1\}$, 
\[\begin{split}
\phi(h_1^{\alpha_1}h_2^{\alpha_2}\cdots h_k^{\alpha_k})
&\in \phi(h_1^{\alpha_1})\phi(h_2^{\alpha_2}\cdots h_k^{\alpha_k})D\\
&\subseteq \ldots\\
&\subseteq \phi(h_1^{\alpha_1})\phi(h_2^{\alpha_2})\cdots \phi(h_k^{\alpha_k})D^{k-1}\\
&\subseteq \phi(h_1)^{\alpha_1}D^2\phi(h_2)^{\alpha_2}D^2\cdots \phi(h_k)^{\alpha_k}D^2D^{k-1}\\
&= \phi(h_1)^{\alpha_1}\phi(h_2)^{\alpha_2}\cdots \phi(h_k)^{\alpha_k}D^{3k-1}.
\end{split}\]

If $0\til \Z^n\til G\til H\til 1$ is a central extension of a finitely generated group $H$ by $\Z^n$,  the extension represents a bounded cohomology class in $H^2(H;\Z^n)$ exactly when the quotient map $\pi\colon G\til H$ has a section $\phi\colon H\til G$ that is a quasimorphism, equivalently, if $D$ as defined above is a finite subset of $\Z^n$.

The following result of A. M. Gersten can now be deduced from Proposition \ref{commuting subgroups} by showing that quasimorphisms between countable discrete groups are bornologous.
\begin{thm}[A. M. Gersten \cite{Gersten}] \label{gersten thm}
If a central extension $0 \to \Z^n \to G \to H \to 1$ of a finitely generated group $H$ represents a bounded cohomology class, then $G$ is quasi-isometric to $H \times \Z^n$.  
\end{thm}

By contrast, not all quasimorphisms, or even homomorphisms, between topological groups are bornologous. For example, a linear map between Banach spaces is bornologous if and only if it is bounded, i.e., continuous. However, one may show that a Baire measurable quasimorphism between Polish groups is bornologous \cite{large scale}.

The following converse to Gersten's theorem appears to be open, although a related result for $L^\infty$-cohomology has been proved in \cite{KL}.  

\begin{quest} \label{converse q}
Let $H$ be a finitely generated group and $0 \to \Z^n \to G \overset{\pi} \longrightarrow  H \to 1$ a central extension.  Suppose $\phi\colon H \times \Z^n \to G$ is a quasi-isometry such that $\pi \phi( h,x)$ is a bounded distance from $h$.   Does $0 \to \Z^n \to G \overset{\pi} \longrightarrow  H \to 1$ represent a bounded cohomology class?  
\end{quest}

Though we do not know the answer to Question \ref{converse q}, and suspect that it may be positive, we will show that the answer is \emph{negative} for the central extensions of (OB) generated groups
$$
0\til \Z^n\to  \Homeo_{\Z^n}(\R^n)\overset{\pi}\longrightarrow \Homeo_0(\T^n) \to 1,
$$
for all $n\geqslant 2$.
By Example \ref{torus example}, there exists a bornologous section $\phi$ of the quotient map $\pi$, and $(k, x)\mapsto k\phi(x)$ is a quasi-isometry between the product $\Z^n\times  \Homeo_0(\T^n)$ and $\Homeo_{\Z^n}(\R^n)$. Hence, to give a negative answer, it suffices to prove the following.  

\begin{thm}\label{QI no QM}
For $n\geqslant 2$, the central extension does not represent a bounded cohomology class in $H^2\big({\rm Homeo}_0(\T^n); \Z^n\big)$, i.e., $\pi$ has no section that is a quasimorphism.
\end{thm}

\begin{proof}
To show that $\pi$ has no section which is a quasimorphism, it will suffice to study sections over commutators.  There is a well known relationship between the set $D$ and values of $\R$-valued quasimorphisms on (products of) commutators originally due to C. Bavard \cite{Bavard}.  As a special case, it is not hard to show if $f\colon  G \to \R$ is a quasimorphism, then the value of $f$ on any commutator is uniformly bounded in terms of $\sup(D)$.   We adapt this line of reasoning to our discussion of sections  of $\Homeo_{\Z^n}(\R^n)\overset{\pi}\longrightarrow \Homeo_0(\T^n)$.

First note that, if $\phi$ and $\psi$ are two sections of $\pi$, then 
$[\phi(a),\phi(b)]=[\psi(a),\psi(b)]$ for all $a,b\in \Homeo_0(\T^n)$. Indeed, given $a,b$, there are $k_a,k_b\in \Z^n$ so that $\psi(a)=\phi(a)k_a$ and $\psi(b)=\phi(b)k_b$. Since $\Z^n$ is central, it follows that
\[\begin{split}
[\psi(a),\psi(b)]
&=\psi(a)\psi(b)\psi(a)\inv \psi(b)\inv\\
&=\phi(a)k_a\phi(b)k_bk_a\inv \phi(a)\inv k_b\inv\phi(b)\inv\\
&=[\phi(a),\phi(b)].
\end{split}\]
Suppose for contradiction that $\phi$ is a section of $\pi$ that is a quasimorphism, i.e., so that 
$$
D=\{  \phi(a) \phi(b)\phi(ab)^{-1} \mid a,b \in \Homeo_0(\T^n) \}^\pm\cup\{1\}
$$ 
is a finite subset of $\Z^n$. Then, $\phi([a,b][c,d])\in [\phi(a),\phi(b)][\phi(c),\phi(d)]D^{23}$ or, equivalently, 
$$
[\phi(a),\phi(b)][\phi(c),\phi(d)]\in \phi([a,b][c,d])D^{23}
$$
for all $a,b,c,d$. In particular, this implies that, for any section $\psi$ of $\pi$, the set
$$
F=\big\{[\psi(a),\psi(b)][\psi(c),\psi(d)]\in \Z^n\mid a,b,c,d\in \Homeo_0(\T^n) \;\&\; [a,b][c,d]=1\big\}
$$
is finite and independent of the choice of $\psi$.

We can now finish the proof of Theorem \ref{QI no QM} using the following example communicated to the first author by Leonid Polterovich.
\begin{exa} \label{Polterovich ex}
For every $N > 0$, there are $a, b, c, d \in \Homeo_0(\T^n)$ with $[a,b][c,d] =1$ so that, for any section $\psi$ of $\pi$, the product 
$$
[\psi(a),\psi(b)][\psi(c),\psi(d)]
$$
has word norm greater than $N$ with respect to the standard generators of $\Z^n$. 

To simplify notation, we will assume that $n=2$; the general case is analogous. 
Let $N >0$   be given.  Since by \cite{EHN} every element of ${\rm Homeo}_+(\T)$ and hence also of ${\rm Homeo}_+(\T)\times {\rm Homeo}_+(\T)$ is a commutator, we can  find
$a, b \in \Homeo_0(\T^2)$ such that $[\psi(a), \psi(b)](x,y) = (x+1/2, y)$ for any choice of $\psi$. 
Let $A \in \SL_2(\Z)$ be the matrix 
$$
A = \left( \begin{smallmatrix} 2N + 1& 1\\ 2N & 1 \end{smallmatrix} \right).
$$  
The linear action of $A$ on $\R^2$ normalizes the set of integral translations, so descends to a homeomorphism of $\T^2$.  Moreover, for any $f \in \Homeo_0(\T^2)$, the induced homomorphism $A f A^{-1}$ of $\T^2$ is also in the identity component $\Homeo_0(\T^2)$.  
Define $c = AaA^{-1}$ and $d = AbA^{-1}$.   Since we are free to re-define $\psi$ if needed, we may assume that $\psi(c)= A\psi(a)A^{-1}$ and $\psi(d) = A \psi(d)A^{-1}$, as these project to $a,b,c$ and $d$ in $\Homeo_0(\T^2)$ respectively.  To be consistent with standard notation, we now write points of $\R^2$ as column vectors to compute:
$$
[\phi(c), \phi(d)]\left( \begin{smallmatrix} x \\ y \end{smallmatrix}  \right) = A[\phi(a), \phi(b)]A^{-1}\left( \begin{smallmatrix} x \\ y \end{smallmatrix}  \right)= A  \big( A^{-1}\left( \begin{smallmatrix} x \\ y \end{smallmatrix}  \right) + \left( \begin{smallmatrix} 1/2 \\ 0 \end{smallmatrix}  \right)  \big) = \left( \begin{smallmatrix} x + N+1/2 \\ y + N  \end{smallmatrix}  \right).
$$
Thus,
$$
[\phi(a), \phi(b)][\phi(c), \phi(d)]\left( \begin{smallmatrix} x \\ y \end{smallmatrix}  \right) = \left( \begin{smallmatrix} x + N +1 \\ y + N  \end{smallmatrix}  \right)
$$
showing that $[\phi(a), \phi(b)][\phi(c), \phi(d)]$ has norm greater than $N$, as desired.   Moreover, since 
$$
[\phi(a), \phi(b)][\phi(c), \phi(d)]
$$ 
is an integral translation, we have
$$
[a,b][c,d] = \pi([\phi(a), \phi(b)][\phi(c), \phi(d)]) = 1\in \Homeo_0(\T^2).
$$
\end{exa}
\end{proof}

\begin{rem} Example \ref{Polterovich ex} can easily be adapted to  also provide  an unbounded class in $H^2(\Homeo_0(\T^n), \Z)$. Indeed, this is represented by the central extension of $\Homeo_0(\T^n)$ by  the group of isotopically trivial homeomorphisms of $\R \times \T^{n-1}$ that commute with integer translation of the $\R$ factor.
\end{rem}

\section{Examples of quasi-isometric embeddings}  \label{embeddings sec}

As a means of further describing the large scale geometry of groups of homeomorphisms, we give examples of  quasi-isometric, isomorphic embeddings between them.  Since $\Homeo_\Z(\R)$ is quasi-isomorphic to $\Z$ (by Corollary \ref{Z cor}), it has the simplest nontrivial quasi-isometry type among all non locally compact groups of homeomorphisms.   Thus, a natural starting point, and our focus here, is producing examples of embeddings of this group into others. 
 
It is easy to construct an injective homomorphism $\Homeo_\Z(\R) \to \Homeo_0(M)$, for any manifold $M$: start with an embedded copy of $\S^{n-1} \times [-1,1]$ in $M$, and define an action of $\Homeo_\Z(\R)$ supported on the interior of this set (which we identity with $\S^{n-1} \times \R$) by $\phi(f) (x, t) = (x, f(t))$.  This extends to an action on $M$ that is trivial on the complement of the embedded $\S^{n-1} \times \R$.   However, such a homomorphism is not a quasi-isometric embedding; in fact, using that $\Homeo(\S^{n-1})$ has property (OB), one can easily show that the image of $\phi$ is \emph{bounded} in $\Homeo_0(M)$.

However, with more work, we are able to construct examples of \emph{undistorted} continuous homomorphisms $\Homeo_\Z(\R) \to \Homeo_0(M)$. 

\begin{prop}
Let $M$ be a manifold with $\dim(M) \geqslant 2$ such that $\pi_1(M)$ has an undistorted element of infinite order.  Then there is a continuous, quasi-isometric isomorphic embedding $\Homeo_\Z(\R) \to \Homeo_0(M)$.  
\end{prop}

\begin{proof}
We start with a special case, defining a quasi-isometric, isomorphic embedding from $\Homeo_\Z(\R)$ to the group of orientation-preserving homeomorphisms of the closed annulus $\A: = \S^1 \times [-1,1]$.  This construction will then be modified so as to have image in the group of homeomorphisms of $\A$ that fix the boundary pointwise.   The case for general $M$ is given by a suitable embedding of a family of such annuli in $M$.  

Define a foliation of $\R^2$ by vertical lines on the lower half-plane, and lines of the form $\{ (x, x-a) \}$ on the upper half plane (figure 1, left).   Each leaf is a set $L_a$ of the form 
$$
\{(a, y) \mid y \leqslant 0\} \cup \{(y+a, y) \mid y \geqslant 0\}
$$
and we have homeomorphism $\psi_a: L_a \to \R$ given by 
$\psi_a(x, y) = y+a$, with inverse 
$$\psi_a^{-1}(z) = \left\{ \begin{array}{cc} 
(a, z-a) & \mbox{ if } z-a<0 \\
(z, z-a) & \mbox{ if } z-a \geqslant 0 \end{array} \right. $$  

This identification lets us define an action of $\Homeo_\Z(\R)$ on $\R^2$, preserving each leaf of the foliation.  On each leaf $L_a$, $f \in \Homeo_\Z(\R)$ acts by $\psi_a^{-1} f \psi_a$.    For $f \in \Homeo_\Z(\R)$, let $\hat{f}(x,y)$ denote the homomorphism of $\R^2$ just described.  Explicitly, if $(x,y) \in L_a$, then 
$$\hat{f}(x,y) = \psi_a^{-1} f (y+a) =  \left\{ \begin{array}{cc} 
(a, f(y+a) -a) & \mbox{ if } f(y+a) -a<0 \\
(f(y+a), f(y+a) -a) & \mbox{ if } f(y+a) -a \geqslant 0 \end{array} \right.$$  

One can check directly that this action commutes with integral horizontal translations, i.e. that $\hat{f}(x+1,y) = \hat{f}(x,y) + (1, 0)$, and so descends to an action of $\Homeo_\Z(\R)$ on the open annulus $\R^2 / (x,y) \sim (x+1, y)$.  






We claim that it also extends to a continuous action on the compactification this quotient to the closed annulus $\A$, giving a homomorphism $\phi\colon \Homeo_\Z(\R) \to \Homeo_0(\A)$.  First, define the compactification explicitly by identifying $\R^2 / (x,y) \sim (x+1, y)$ with $\R \times (-1, 1) /  (x,y) \sim (x+1, y)$ via the map 
$\R^2 \to \R \times (-1, 1)$ given by $(x, y) \mapsto (x, \frac{2}{\pi} \arctan(y))$  
and compactify by adding the boundary circles $\S^1 \times \{-1\}$ and $\S^1 \times \{1\}$.  Figure 1 (right) shows the foliation on this annulus; leaves are transverse to the circle $\S^1 \times \{-1\}$ and spiral around the circle $\S^1 \times \{1\}$
Now it is easy to check explicitly that the action of $\Homeo_\Z(\R)$ extends continuously to a trivial action on the $\S^1 \times \{-1\}$ boundary component, and to the standard action of $\Homeo_+(\S^1) = \Homeo_\Z(\R)/\Z$ on the $\S^1 \times \{1\}$ boundary component.    

\begin{figure}
\centering
\includegraphics[width=3.8in]{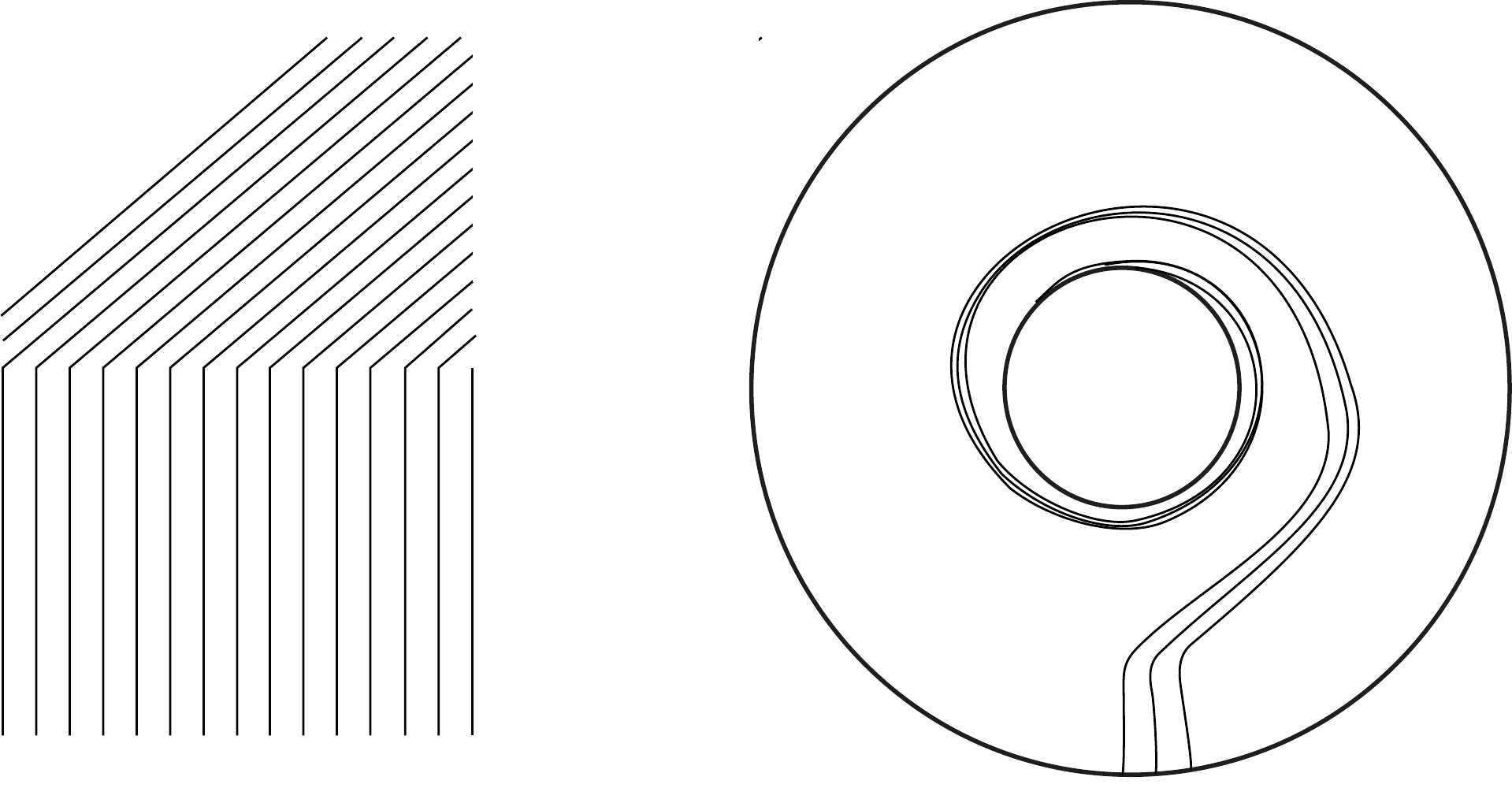} 
\caption{A foliation of the plane and induced foliation on the annulus}
\end{figure}

To show that this is a quasi-isometric embedding,  let $U$ be a symmetric relatively (OB) neighbourhood of the identity in $\Homeo(\A)$ and let $\norm\cdot_U$ denote the corresponding length function on $\Homeo(\A)$. Also, for $n\in \Z$,  let $T_n \in \Homeo_\Z(\R)$ denote translation by $n$.  Since the inclusion $n \mapsto T_n$ of $\Z$ into $\Homeo_\Z(\R)$ is a quasi-isometry between $\Z$ and $\Homeo_\Z(\R)$, it suffices to verify that $ n \leqslant K \norm{\phi(T_n)}_U + K$ for some constant $K$.  

Let $\tilde{T}_n$ be the lift of $\phi(T_n)$ to the universal cover $\tilde{\A} = \R \times [-1, 1]$ of $A$ that fixes $\R \times \{-1\}$ and acts on $\R \times \{1\}$ as translation by $n$.  It follows that $\tilde{T}_n([0,1] \times [-1,1])$ has diameter at least $n$.  Lemma \ref{militon lemma} now gives the desired inequality.  

By gluing two copies of $\A$ together on the $\S^1 \times \{1\}$ boundary component, and ``doubling" the action defined above, we get an action of $\Homeo_\Z(\R)$ on the closed annulus that fixes each boundary component pointwise.  Translations in $\Homeo_\Z(\R)$ act by rotations on the essential $\S^1$ curve corresponding to the identified $\S^1 \times \{1\}$ boundary components, so the same argument as above shows that this is also a quasi-isometric embedding.

This gluing construction easily generalizes to higher dimensions.  Rather than gluing two copies of $\A$ together on the $\S^1 \times \{1\}$ boundary, we can take a family of copies of $\A$ parametrized by $\S^{n-2}$, and identify the $\S^1 \times \{1\}$ boundary component of each annulus in the family with a single copy of $\S^1$.  In this way we get an action of  $\Homeo_\Z(\R)$ on $\S^1 \times \D^{n-1}$ that fixes the boundary pointwise.  As before, translations in $\Homeo_\Z(\R)$ act as point-pushes along an essential $\S^1 \times \{0\} \subset \S^1 \times \D^{n-1}$ curve.  

For the general case, suppose that $M$ is a manifold of dimension $n \geqslant 2$ such that $\pi_1(M)$ has an undistorted infinite order element.   As in the proof of Proposition \ref{C prop 1}, we may find an embedded curve $\gamma \subseteq M$ representing such an element.  Identify $\S^1 \times \D^{n-1}$ with a tubular neighborhood of $\gamma$ so that $\gamma = \S^1 \times \{0\}$.  The construction above now gives an action of $\Homeo_\Z(\R)$ on this tubular neighbourhood and fixing the boundary, so it extends (trivially) to an action of $\Homeo_\Z(\R)$ by homeomorphisms of $M$.  As in Proposition \ref{C prop 1}, the action of $T_n$ ``point-pushes" along $\gamma$, and so the estimate from Lemma \ref{militon lemma} given there shows that this is a quasi-isometric, isomorphic embedding.
\end{proof}


\section{Application to group actions and distorted subgroups}  \label{group actions sec}

It is a basic, although typically very difficult, problem to classify the actions of a given group $G$ on a manifold $M$, i.e. the homomorphisms $G \to \Homeo(M)$.  In the case where $G$ is a finitely generated discrete group, progress on this problem has been made in three separate cases, all under strong additional assumptions.  The first case is the study of actions by diffeomorphisms preserving a volume form on $M$.  In the case where $G$ is assumed to be a lattice in a Lie group, this is the classical \textit{Zimmer program}.  The second tractable case is the very special case of $M = \S^1$, and the third case is centered around the theme of using torsion elements in $G$ to show that $G$ cannot act on a manifold.  See \cite{Fisher} for a survey of results in all three directions.  

Other than these special cases, the problem of classifying (discrete) group actions on manifolds remains wide open.  
For example, the following statement appears in a recent paper of Fisher and Silberman \cite{FS}.
\begin{quote}
\textit{Except for $M = \S^1$, there is no known example of a torsion-free finitely generated group that doesn't act by homeomorphisms on $M$.}
\end{quote}
If $G$ is a locally compact, rather than finitely generated, group, again very little can be said about actions of $G$ by homeomorphisms save for the special cases where $G$ is compact or has many torsion elements.  
\smallskip

We propose that one should instead study homomorphisms $G \to \Homeo_0(M)$ that respect the geometric and topological structure of $G$, namely, the quasi-isometric, isomorphic embeddings.

\begin{prob}  \label{embed problem}
Let $G$ be a compactly generated group, and $M$ a manifold.  Classify the quasi-isometric, isomorphic embeddings $G \to \Homeo_0(M)$.  
\end{prob} 
This problem should be most accessible in the case where $M$ is a compact surface, and, as a first step, one should produce examples of groups with \emph{no} quasi-isometric, isomorphic embeddings into $\Homeo_0(M)$, and/or groups with only very few or very rigid QI isomorphic embeddings.


\subsection{Groups that admit no Q.I. embedding into $\Homeo_0(M)$} 

The goal of this section is to provide evidence that Problem \ref{embed problem} is {\em a)} approachable, and {\em b)} essentially different from the general problem of classifying group actions on manifolds.  To do this, we construct an example of a compactly generated group that embeds continuously in a group of homeomorphisms, but only admits distorted continuous embeddings.   

Let $\Homeo_0(\A)$ denote the identity component of the group of homeomorphisms of the closed annulus.  These homeomorphisms preserve boundary components, but need not pointwise fix the boundary.  Our main proposition is the following.    
 
\begin{prop} \label{no QI}
There exists a torsion-free, compactly generated, closed subgroup $G$ of  $\Homeo_0(\A)$ such that $G$ does not admit any continuous quasi-isometric, isomorphic embedding into $\Homeo_0(\A)$.  
\end{prop}

The obstruction to a continuous quasi-isometric, isomorphic embedding given in our proof comes from a combination of Theorem \ref{militon thm} and some basic results from the theory of rotation vectors for surface homeomorphisms, as developed by Franks \cite{Franks} and Misiurewicz and Ziemian \cite{MZ}.    

\begin{proof}
Let $G = \Z \ltimes \R$, where the action of the generator $z$ of $\Z$ on any $r \in \R$ is given by $z r z\inv = -r$.   (Here and in what follows, we think of $\Z$ as the cyclic group generated by $z$, and $\R$ as an additive group.)  $G$ is quasi-isometric to $\Z \times \R$, as the subgroup $2\Z \times \R \leqslant G$ is of index 2.  
We first produce a continuous, injective homomorphism from $G$ to $\Homeo_0(\A)$.  In fact, the same strategy can be used to produce a continuous injective homomorphism $G \to \Homeo_0(M)$, for any manifold $M$ of dimension at least 2.

Let $B \subseteq \A$ be an embedded ball, and let $g \in \Homeo_0(\A)$ be a homeomorphism such that the translates $g^k(B)$, for $k \in \Z$, are pairwise disjoint.  Let $\{f_t\}_{t\in \R}$ be a 1-parameter family of homeomorphisms supported on $B$.  For $r \in \R$ define 
$$
\phi(r)(x)=\begin{cases}
g^k f_r g^{-k}(x) &\text{ if } x \in g^k(B), k \text { even} \\
g^k f_{-r} g^{-k}(x) &\text{ if } x \in g^k(B), k \text { odd} \\
x &\text{ otherwise }.
\end{cases}
$$ 
and define $\phi(z) = g$.  Then $\phi(z)\phi(r)\phi(z)^{-1} = \phi(r)^{-1}$ and so $\phi$ extends to a homomorphism defined on $G$. By construction, this homomorphism is continuous and can easily be insured to be an isomorphic embedding.

Now we show that there is no continuous quasi-isometric, isomorphic embedding of $G$ into $\Homeo_0(\A)$.  
To see this, suppose for contradiction that $\psi$ was such an embedding.   We focus first on the $\R$ factor of $G$.  The first goal is to show that the algebraic constraint given by undistortedness of $\psi(\R)$ puts a dynamical constraint on the action of $\psi(\R)$.  
For concreteness, identify the universal cover $\tilde{\A}$ of $\A$ with $\R \times [0,1]$, and the covering map with the standard projection $\pi: \R \times [0,1] \to \R/\Z \times [0,1] \cong \A$.  
Let $d$ be the metric on $\A$ induced by the standard Euclidean metric $\tilde{d}$ on $\R \times [0,1]$, so $\tilde{d}$ is a geodesic lift of $d$. 

Continuity of $\psi$ together with the construction of lifts of paths of homeomorphisms described in Section \ref{extension sec} gives a lift of the action of $\R$ on $\A$, i.e. a homomorphism $\tilde{\psi}: \R \to \Homeo(\tilde{\A})$  such that $\pi \tilde{\psi} = \psi$.

Let $t$ be any nontrivial element of $\R \leqslant G$.  Since $\langle t \rangle \leqslant \R \leqslant G$ is 
quasi-isometrically embedded, the restriction of $\psi$ to $\langle t \rangle$ is a quasi-isometric, isomorphic embedding into $\Homeo_0(\A)$.   
To simplify notation, let $\tau = \tilde{\psi}(t)$, so $\tau^n = \tilde{\psi}(nt)$, and let $D = [0,1] \times [0,1]$ be a fundamental domain for $\A$.  
Theorem \ref{militon thm} implies that there exist constants $K', C'$ such that
\begin{equation}
n \leqslant K' \sup_{x,y \in D} \tilde{d}(\tau^n(x),   \tau^n(y)) + C'
\end{equation}
holds for all $n$.  In particular, for $K:=2K'$ and $C:=C'+K'\cdot {\rm diam}(D)$, we have
\begin{equation}
n \leqslant K \sup_{x \in D} \tilde{d}(\tau^n(x),  x) + C
\end{equation}
for all $n$.  Increasing $K$ and $C$ further if needed, we easily get the other inequality
\begin{equation} \label{mil eq}
\frac 1K \sup_{x \in D} \tilde{d}(\tau^n(x),  x) - C\leqslant n \leqslant K \sup_{x \in D} \tilde{d}(\tau^n(x),  x) + C.
\end{equation}

Using an argument from \cite{MZ}, we will now find a single point $x$ that achieves such a ``linear displacement"  under $\tau^n$ for all $n$.   Combined with a theorem of Franks, this will give the existence of a periodic orbit for $\psi(\R)$.  
As this argument uses tools separate from the rest of the proof, we state it as an independent lemma.  

\begin{lemma} \label{MZ lem}
Suppose $\tau \in \Homeo(\tilde{\A})$ is a lift of an annulus homeomorphism satisfying the inequalities in \eqref{mil eq} above.  Then  
there exists a point $x \in D$ such that  
$$\lim \limits_{n \to \infty} \frac{1}{n} \tilde{d}(\tau^n(x), x) = K$$
\end{lemma} 

\begin{proof}  
Pick a sequence of points $x_n \in D$ such that $\tilde d(\tau^n(x_n),x_n)\geqslant \sup_{x \in D} \tilde{d}(\tau^n(x),  x)-\frac 1n$. By Inequality \ref{mil eq}, we have 
$$
n/K-C/K-\frac 1n\leqslant \tilde d(\tau^n(x_n),x_n)
\leqslant Kn+KC.
$$
This implies that there exists $0< L \leq K$, and a subsequence $x_{n_k}$ such that  
$$
\lim \limits_{k \to \infty} \frac{\tau^{n_k}(x_{n_k}) - x_{n_k}}{n_k} = (L, 0)
$$
where $\tau^n(x_n) - x_n$ is considered as a vector in $\R \times [0,1] \subseteq \R^2$.  

Since, the estimate $
\limsup \limits_{n \to \infty} \frac{\tau^n(y_n) - y_n}{n} \leqslant (K, 0)
$
holds for any sequence of points $y_n$, we may take $x_{n_k}$ to be a sequence such that maximizes the value of the norm of $\lim \limits_{k \to \infty} \frac{\tau^{n_k}(y_{n}) - y_{n}}{n}$ over all sequences in $D$ such that this limit exists.   

In \cite[Theorem 2.4]{MZ}, it is shown that, under these hypotheses, there exists a probability measure $\mu$ on $D$, invariant under $\tau$, and such that $\int_D (\tau(x) -x) \, d\mu = (L,0)$.   (The result of \cite{MZ} is in fact more general -- applying not only to annulus homeomorphisms, but to continuous non-invertible maps, and to continuous maps of the torus.  In their language, $(L, 0)$ is an extremal point of the \emph{rotation set} for $\tau$.)  
It now follows from the Birkhoff ergodic theorem that, for $\mu$-almost every $x \in D$, 
$$\lim \limits_{n \to \infty} \frac{\tau^n(x) - x}{n} = \int_D (g(x) - x)\,  d\mu = (L,0)$$
which give the existence of a point as in the Lemma.  
Further details can be found in \cite[Section 1]{MZ}.  
\end{proof}

Continuing with the proof of the Theorem, let $s = t/L$, so $\tilde{\psi}(s)^{L} = \tilde{\psi}(t)$.  This makes sense even for non-integer $L$, as $s$ is in the (additive) $\R$ subgroup, so we may \emph{define} $\tilde{\psi}(s)^{L}$ to be $\tilde{\psi}(Ls)$.
From  Lemma \ref{MZ lem} we have a point $x$ such that 
$$\lim \limits_{n \to \infty} \frac{\tilde{\psi}(t)^n(x) - x}{n} = (L, 0)
$$
and this implies that
$$\lim \limits_{n \to \infty}  \frac{\tilde{\psi}(s)^n(x) - x}{n} = (1,0).
$$

Let $T: \R \times [0,1] \to \R \times [0,1]$ be the translation $T(a,b) = (a+1, b)$.  Since $T$ commutes with $s$, we have 
$$\lim \limits_{n \to \infty}  \frac{(T^{-1}\circ \tilde{\psi}(s))^n(x) - x}{n} = (0,0).
$$

By work of Franks \cite[Cor. 2.5]{Franks}, this implies that $T^{-1} \circ \tilde{\psi}(s)$ has a fixed point. (In the language of \cite{Franks},  $T^{-1}\circ \tilde{\psi}(s)$ is said to have a periodic point with rotation number 0; this simply means a fixed point.)   Equivalently, there exists a point $x_0 \in \R \times [0,1]$ such that $\tilde{\psi}(s)(x_0) = x_0 + (1, 0)$.   Since $\tilde{\psi}(\R)$ commutes with $T$, we have also $\tilde{\psi}(r+s)(x_0) = \tilde{\psi}(r)(x_0) + (1, 0)$ for all $r \in \R$.    It follows that the orbit of $\pi(x_0) \in \A$ under $\psi(\R)$ is an embedded circle, homotopic to a boundary component of the annulus.   Let $C$ denote this embedded circle.   

As before, let $z$ denote a generator of the $\Z$ factor of $G = \Z \ltimes \R$.  Since $zrz^{-1} = -r$, the embedded circle $\psi(z)(C)$ is the orbit of the point $\psi(z)(x_0)$ under $\psi(\R)$.   From the group relation and our construction of the lift $\tilde{\psi}$ by lifting paths, it follows that if $x \in \R \times [0,1]$ satisfies $\pi(x) \in \phi(z)(C)$, then we have
\begin{equation} \label{neg eq}
\tilde{\psi}(s)(x)= x + (-1, 0).
\end{equation} 
This implies in particular that $\psi(z)(x_0) \neq x_0$, and that the curves $\psi(z)(C)$ and $C$ are disjoint.  

As both $\psi(z)(C)$ and $C$ are homotopic to boundary circles, $\psi(z)(C)$ separates $C$ from one of the boundary components of $\A$.  For concreteness, assume this is the boundary component $\S^1 \times \{0\}$.  If $A_0$ is the closed annulus bounded by $\S^1 \times \{0\}$ and $C$, then $\psi(z)(A_0)$ is a sub-annulus contained in the interior of $A_0$ whose boundary is $\S^1 \times \{0\} \sqcup \psi(z)(C)$.   It follows that the set $\{\psi(z)^n(A_0) \mid n \geq 0\}$ is a family of nested annuli, and the boundary components $\psi(z)^n(C)$ of each are pairwise disjoint curves, with $\psi(z^n)(C)$ separating $\psi(z^{n-1})(C)$ from $\psi(z^{n+1})(C)$.  

Since $\A$ is compact, there exists an accumulation point $y$ of the curves $\psi(z^n)(C)$.  Since $\psi(z^n)(C)$ separates $\psi(z^{n-1})(C)$ from $\psi(z^{n+1})(C)$, $y$ is an accumulation point both of the set of curves of the form $\psi(z^{2m})(C)$ and of the set of curves of the form $\psi(z^{2m+1})(C)$.     We now use this to derive a contradiction -- informally speaking, the relation $zrz^{-1} = -r$ will force the flow given by $\psi(\R)$ to move points in opposite directions on the curves $\psi(z^{2m})(C)$ and $\psi(z^{2m+1})(C)$, and so it cannot be continuous at $y$.  

To make this more formal, let $\tilde{y}$ be a lift of $y$ to $\tilde{\A} = \R \times [0,1]$.  Then there exists a sequence $a_n$ in $\tilde{\A}$ approaching $\tilde{y}$ and such that $\pi(a_n) \in \psi(z^n)(C)$.  
Since $a_n \in \psi(z^n)(C)$, using the group relation as in equation \eqref{neg eq} above, we have 
$$\tilde{\psi}(s)(a_n) = 
\left\{ \begin{array}{ll} a_n + (1, 0) & \text{ if } n = 2m+1 \\
a_n + (-1, 0) & \text{ if } n = 2m
\end{array} \right.$$
This contradicts continuity of the homeomorphism $\tilde{\psi}(s)$ at $y$, concluding the proof.  
\end{proof}

Although flows played a central role in the proof of Propositions \ref{no QI}, we expect that it should be possible to give discrete examples.  
\begin{quest}
Can one modify the construction above to give an example of a torsion-free, finitely generated subgroup of $\Homeo_0(\A)$ that does not quasi-isometrically isomorphically embed into $\Homeo_0(\A)$?   
\end{quest}

A good candidate for this question might be the subgroup $\Z \ltimes \Z \leqslant \Z \ltimes \R$, or the infinite dihedral group $\Z/2\Z \ltimes \Z$.  

It would also be interesting to see analogous constructions on other manifolds.  The tools we used from \cite{MZ} and \cite{Franks} generalize directly to homeomorphisms of the torus, and less directly to other surfaces, making the 2-dimensional case quite approachable.


\end{document}